\documentclass{amsart}
\usepackage{amssymb}

\newcommand{\map}[1]{\xrightarrow{#1}}

\newcommand{\mil}{\varprojlim}

\newcommand{\iso}{\cong}
\newcommand{\define}{\stackrel{\mathrm{def}}{=}}

\newcommand{\Gal}{\mathrm{Gal}}
\newcommand{\Hom}{\mathrm{Hom}}
\newcommand{\Aut}{\mathrm{Aut}}
\newcommand{\End}{\mathrm{End}}
\newcommand{\Q}{\mathbb Q}
\newcommand{\Z}{\mathbb Z}

\newcommand{\C}{\mathbb C}
\newcommand{\F}{\mathbb F}
\newcommand{\wild}{\mathrm{wild}}
\newcommand{\tame}{\mathrm{tame}}
\newcommand{\hecke}{\mathfrak{h}}

\newcommand{\co}{\mathcal{O}}
\newcommand{\Ta}{\mathrm{Ta}}
\newcommand{\fn}{\mathfrak{n}}
\newcommand{\ord}{\mathrm{ord}}
\newcommand{\cyc}{\mathrm{cyc}}
\newcommand{\FT}{\mathbf{T}}
\newcommand{\fa}{\mathfrak{a}}
\newcommand{\Sel}{\mathrm{Sel}}
\newcommand{\Gr}{\mathrm{Gr}} 

\newcommand{\Iw}{\mathrm{Iw}}
\newcommand{\unr}{\mathrm{unr}}
\newcommand{\Spec}{\mathrm{Spec}}
\newcommand{\Frob}{\mathrm{Fr}}
\newcommand{\tors}{\mathrm{tors}}
\newcommand{\ram}{\mathrm{ram}}

\input xy
\xyoption{all}
\begin{document}

\author{Benjamin Howard}

\address{Dept. of Mathematics, University of Chicago,
5734 S. University Ave., Chicago, IL 60637}

\email{howardbe@bc.edu}

\title{Variation of Heegner points in Hida families}

\begin{abstract}
Given a weight two modular form $f$ with associated $p$-adic Galois representation $V_f$, for certain quadratic imaginary fields $K$ one can construct canonical classes in the Galois cohomology of $V_f$ by taking the Kummer images of Heegner points on the modular abelian variety attached to $f$.  We show that these classes can be interpolated as $f$ varies in a Hida  family and construct an Euler system of big Heegner points for Hida's universal ordinary deformation of $V_f$.  We show that the specialization of this big Euler system to any form in the Hida family is nontrivial, extending results of Cornut and Vatsal from modular forms of weight two and trivial character to all ordinary modular forms, and propose a  horizontal nonvanishing conjecture for these cohomology classes.  The horizontal nonvanishing conjecture  implies, via the theory of Euler systems,  a  conjecture of Greenberg on the generic ranks of Selmer groups in Hida families. 
\end{abstract}

\maketitle

\theoremstyle{plain}
\newtheorem{Thm}{Theorem}[subsection]
\newtheorem{Prop}[Thm]{Proposition}
\newtheorem{Lem}[Thm]{Lemma}
\newtheorem{Cor}[Thm]{Corollary}
\newtheorem{Con}[Thm]{Conjecture}
\newtheorem{BigTheorem}{Theorem}

\theoremstyle{definition}
\newtheorem{Def}[Thm]{Definition}
\newtheorem{Hyp}[Thm]{Hypothesis}

\theoremstyle{remark}
\newtheorem{Rem}[Thm]{Remark}
\newtheorem{Ques}[Thm]{Question}

\renewcommand{\labelenumi}{(\alph{enumi})}
\renewcommand{\theBigTheorem}{\Alph{BigTheorem}}

\section{Introduction}

Fix a positive integer $N$ and a prime $p\nmid N$, and fix, once and for all, embeddings of algebraic closures  $\overline{\Q}\hookrightarrow\overline{\Q}_p$, $\overline{\Q}\hookrightarrow\C$.   We denote by $\omega:(\Z/p\Z)^\times\map{}\mu_{p-1}$ the Teichmuller character and view $\omega$ also as a Dirichlet character modulo $Np$.  Let
$$
g=\sum_{n>0}a_nq^n \in S_{k}(\Gamma_0(Np),\omega^j)
$$
be a normalized eigenform (for all Hecke operators $T_\ell$ for $\ell\nmid Np$ and $U_\ell$ for $\ell\mid Np$) of weight $k\ge 2$ and character $\omega^j$. The existence of such a form implies $j\equiv k\pmod{2}$. Fix a  finite extension $F/\Q_p$ which contains all Fourier coefficients of $g$ and let $\co_F$ denote the ring of integers of $F$. We assume that $g$ is an \emph{ordinary $p$-stabilized newform} in the sense that $a_p\in \co_F^\times$ and the conductor of $g$ is divisible by $N$, (i.e. the system of Hecke eigenvalues $\{a_n\mid (n, Np)=1\}$ of $g$ agrees with that of a new eigenform of level $N$ or $Np$).  Let $\rho_g:G_\Q\map{}\mathrm{GL}_2(F)$ be the $p$-adic Galois representation attached to $g$ by Deligne. Fix a quadratic imaginary field $K$.

\begin{Hyp}
The data of the previous paragraph is to remain fixed throughout this article, and the following hypotheses are assumed throughout:
\begin{enumerate}
\item
$p\nmid 6N$;
\item
there is an ideal $\mathfrak{N}$ of the maximal order $\co$ of $K$ such that $\co/\mathfrak{N}\iso \Z/N\Z$;
\item
the semi-simple residual representation attached to $\rho_g$ is absolutely irreducible.
\end{enumerate}
\end{Hyp}

Note that we impose no hypotheses on the behavior of $p$ in $K$; $p$ may be split, ramified, or inert.
For the remainder of the introduction we will also assume that $p\nmid\phi(N)$ (Euler's function) and that $N$ is relatively prime to $\mathrm{disc}(K)$, so that the existence of the ideal $\mathfrak{N}$ implies that all prime divisors of $N$ are split in $K$.

In \S \ref{Hida} we recall Hida's definition of a local domain $R$, finite and flat over the Iwasawa algebra $\Lambda=\co_F[[1+p\Z_p]]$, whose arithmetic prime ideals parametrize  the \emph{Hida family} of $g$. We also recall Hida's construction of a big Galois representation $\FT$ which is free of rank two over $R$ and admits a twist $\FT^\dagger$ possessing a perfect alternating pairing $\FT^\dagger\times\FT^\dagger\map{}R(1)$. The main result of this article is the construction, for every positive integer $c$  prime to $N$, of a canonical cohomology class
$$
\mathfrak{X}_c\in  \tilde{H}^1_f(H_c,\FT^\dagger)
$$
where $H_c$ is the ring class field of $K$ of conductor $c$ and $\tilde{H}_f^1(H_c,\FT^\dagger)$ is  Nekov{\'a}{\v{r}}'s extended Selmer group.  As the conductor $c$ varies the classes $\mathfrak{X}_c$ form an Euler system in the sense of Kolyvagin.  The construction and basic properties of the classes $\mathfrak{X}_c$ are given in \S \ref{ss:construction}-\ref{ss:selmer}.

To any arithmetic prime ideal $\mathfrak{p}\subset R$  Hida's theory associates an ordinary modular form $g_\mathfrak{p}$ with coefficients in $F_\mathfrak{p}$, the residue field of the localization $R_\mathfrak{p}$ (so that $F_\mathfrak{p}$ is a finite extension of $\Q_p$).  If we define a Galois representation $V_\mathfrak{p}^\dagger= \FT^\dagger\otimes_R F_\mathfrak{p}$ then $V_\mathfrak{p}^\dagger$ is a self-dual twist of the $p$-adic Galois representation attached to $g_\mathfrak{p}$ by Deligne.  Applying the map on cohomology induced by $\FT^\dagger\map{}V_\mathfrak{p}^\dagger$   to the classes $\mathfrak{X}_c$ we obtain an Euler system for $V_\mathfrak{p}^\dagger$.  In the case of weight $2$ and trivial character this construction essentially recovers the Kummer images of classical Heegner points on modular abelian  varieties.  In \S \ref{ss:vnt} and \S \ref{ss:prnv} we  show that the image of $\mathfrak{X}_{p^s}$ in $\tilde{H}^1_f(H_{p^s},V_\mathfrak{p}^\dagger)$ is nontrivial for $s\gg 0$.  The precise result, Corollary \ref{VNV}, is actually somewhat stronger and  extends Cornut and Vatsal's proof of Mazur's conjecture on the nonvanishing of Heegner points from the case of modular forms of  weight two and trivial character to all ordinary modular forms. 
Given this nonvanishing result, a suitable extension of Kolyvagin's theory of Euler systems  would prove that 
\begin{equation}\label{growth}
\mathrm{dim}_{F_\mathfrak{p}} \tilde{H}^1_f(D_s, V_\mathfrak{p}^\dagger) =p^s+O(1)
\end{equation}
for every form  $g_\mathfrak{p}$ in the Hida family.  Here $D_s$ is the subfield of the anticyclotomic $\Z_p$-extension of $K$  (i.e. the unique $\Z_p$-extension contained in $\cup H_{p^s}$) having degree $p^s$ over $K$.  The weaker result that (\ref{growth}) holds for all but finitely many  $g_\mathfrak{p}$, and for all $g_\mathfrak{p}$ of weight $2$, has been proved by Nekov{\'a}{\v{r}}   (\cite{selmer_complexes} Theorems 12.9.11(ii)  and 12.9.8(i), respectively).

In \S \ref{SS:MC} and \S \ref{SS:HNV} we propose two conjectures.  The first, Conjecture \ref{mc}, is a two-variable extension of Perrin-Riou's \cite{pr}  Iwasawa main conjecture for Heegner points.  As above, a suitable extension of the theory of Euler systems would prove one divisibility of this conjecture.  See  \cite{bert, HowA, HowB}  for results toward Perrin-Riou's original conjecture.  The second conjecture, Conjecture \ref{HNV}, is that the corestriction of $\mathfrak{X}_1$ from $H_1$ to $K$ is not $R$-torsion.  This conjecture implies, by an extension of Kolyvagin's theory due to  Nekov{\'a}{\v{r}},  a conjecture of Greenberg  \cite{greenberg_deformation}  predicting that the $F_\mathfrak{p}$ dimension of the Selmer group $\tilde{H}^1_f(\Q,V_\mathfrak{p}^\dagger)$ is equal to zero or one (depending on the sign of the functional equation of the Hida family of $g$) for all but finitely many  $g_\mathfrak{p}$ in the Hida family.   See Corollary \ref{generic rank}.

We remark that higher weight analogs of Heegner points have been constructed elsewhere in the literature, e.g.  \cite{howard-cycles,nek:euler, nek:gz, zhang97}, using  special cycles on Kuga-Sato varieties.   Our method is completely different.  For ordinary modular forms of even weight and trivial character both constructions are valid, and in these cases it would be interesting to understand the connection between the two.

The author wishes to express his thanks to J. Pottharst for pointing out the ambiguity in $\Theta$ discussed in Remark \ref{twist remark}, and to the anonymous referee for providing helpful comments on an earlier version of this article.

Throughout the article Galois cohomology is always understood to mean continuous cohomology.

%%%%%%%%%%%%%%%%%%%%%%%%%%%%%%%%%%%%%%%%%%%%%%%%%%%%%%%%%%%%%%%%%%%%%

\section{Big Heegner points}
\label{S:construction}

%%%%%%%%%%%%%%%%%%%%%%%%%%%%%%%%%%%%%%%%%%%%%%%%%%%%%%%%%%%%%%%%%%%%%%%

Set $\Phi_s=\Gamma_0(N)\cap\Gamma_1(p^s)\subset\mathrm{SL}_2(\Z)$ and let $Y_s$ denote the affine modular curve classifying  elliptic curves with $\Phi_s$ level structure, by which we mean a triple consisting of an elliptic curve $E$, a cyclic order $N$ subgroup of $E$, and a point of exact order $p^s$ on $E$.  Let $Y_s\hookrightarrow X_s$ be the usual  compactification obtained by adjoining cusps and let $J_s$ be the Jacobian of $X_s$.   Denote by
$$
X_{s+1}\map{\alpha}X_s
$$
the degeneracy map which is given by $(E,C,\pi) \mapsto (E,C, p\cdot \pi)$ on the affine curve $Y_{s+1}$. We view $Y_s$, $X_s$, and $J_s$ as schemes over $\Spec(\Q)$.

%%%%%%%%%%%%%%%%%%%%%%%%%%%%%%%%%%%%%%%%%%%%%%%%%%%%%%%%%%%%%%%%%

\subsection{Hida theory}
\label{Hida}

%%%%%%%%%%%%%%%%%%%%%%%%%%%%%%%%%%%%%%%%%%%%%%%%%%%%%%%%%%%%%%%

We recall the basic facts of Hida theory that we need; the reader may refer to \cite{EPW,GS,hida,selmer_complexes} for more details.  Identify $\mu_{p-1}$ with $(\Z/p\Z)^\times$ using the Teichmuller character and abbreviate
$$
\Gamma=1+p\Z_p
\hspace{1cm}
\Delta=(\Z/p\Z)^\times
$$
so that $\Z_p^\times \iso \Delta\times\Gamma.$ Define the Iwasawa algebra $\Lambda=\co_F[[\Gamma]]$ and write $z\mapsto [z]$ for the inclusion of group-like elements $\Z_p^\times \map{}\co_F[[\Z_p^\times]]^\times$. For each  $i\in \Z/(p-1)\Z$ define an idempotent $e_i\in\co_F[[\Z_p^\times]]$ by
$$
e_i=\frac{1}{p-1}\sum_{\delta\in\Delta } \omega^{-i}(\delta)[\delta]
$$
Let  $\co_F[[\Z_p^\times]]\map{}\hecke^\ord$ be Hida's big ordinary Hecke algebra of tame level $N$, defined as follows. Let $\hecke_{r,s}$ be the $\co_F$-algebra generated by all Hecke operators $T_\ell$ for $\ell\nmid Np$, together with the operators $U_\ell$ for $\ell\mid Np$ and the nebentype operators $\langle m\rangle$ for  $m\in (\Z/p^s\Z)^\times$, acting on the space of $p$-adic cusp forms $S_r(\Phi_s, \overline{\Q}_p)$. We make $\hecke_{r,s}$ into an $\co_F[[\Z_p^\times]]$-algebra by  $[z] \mapsto z^{r-2}\langle z\rangle$ (this normalization differs from much of the literature, in which  $[z]\mapsto z^r\langle z\rangle$; our normalization is chosen so that the action of $[z]$ agrees with $\langle z\rangle$ in weight two).  Note that $[-1]\mapsto 1$ as $\langle -1\rangle$ acts as  $(-1)^r$ on modular forms of weight $r$.  Hida's ordinary projector $e^\ord=\lim U_p^{m!}$ defines an  idempotent in each $\hecke_{r,s}$, and these are  compatible with the natural surjections $\hecke_{r,s+1}\map{}\hecke_{r,s}$.   If we define $\hecke_{r,s}^\ord=e^\ord\hecke_{r,s}$ then the  algebra 
$$
\hecke^\ord=\mil_s \hecke_{r,s}^\ord
$$
is finite and flat over $\Lambda$ and is independent of the weight $r$ by \cite[Theorem 1.1]{hida}.

\begin{Def}\label{arithmetic}
If $A$ is any finitely generated commutative $\Lambda$-algebra then an $\co_F$-algebra map $A\map{}\overline{\Q}_p$ is \emph{arithmetic}  if the composition
$$
\Gamma\map{\gamma\mapsto[\gamma]}A^\times\map{}\overline{\Q}_p^\times
$$
has the form $\gamma\mapsto \psi(\gamma)\gamma^{r-2}$ for some integer $r\ge 2$ and some finite order character $\psi$ of $\Gamma$. The kernel of an arithmetic map is  an \emph{arithmetic prime} of $A$.  If $\mathfrak{p}$ is an arithmetic prime then the residue field $F_\mathfrak{p}=A_\mathfrak{p}/\mathfrak{p}A_\mathfrak{p}$ is a finite extension of $F$.   The composition 
$\Gamma\map{}A^\times\map{}F_\mathfrak{p}^\times$
has the form $\gamma\mapsto \psi_\mathfrak{p}(\gamma)\gamma^{r-2}$ for a finite order character $\psi_\mathfrak{p}:\Gamma\map{}F_\mathfrak{p}^\times$ called the \emph{wild character} of $\mathfrak{p}$ and an integer $r$ called the \emph{weight} of $\mathfrak{p}$.
\end{Def}

Our fixed cuspform $g$ determines an  arithmetic  map  (denoted the same way) 
$$
g:\hecke^\ord\map{} \hecke^\ord_{k,1}\map{}\co_F
$$ 
characterized by $T_\ell\mapsto a_\ell$ for $\ell\nmid Np$, $U_\ell\mapsto a_\ell$ for $\ell\mid Np$, and 
$$
[\delta] \mapsto \omega^{k+j-2}(\delta)
\hspace{1cm}
[\gamma ]\mapsto \gamma^{k-2}
$$
for $\delta\in \Delta$ and $\gamma\in \Gamma$. There is a decomposition  of $\hecke^\ord$ as a direct sum of its completions at maximal ideals, and we let $\hecke^\ord_\mathfrak{m}$ be  
the unique local summand through which $g$ factors.   
As $g(e_i)=0$ for $i\not=k+j-2$, we must have
$$
\hecke^\ord_\mathfrak{m}=e_{k+j-2} \hecke^\ord_\mathfrak{m}.
$$ 
According to \cite[\S 12.7.5]{selmer_complexes} the localization of $\hecke_{\mathfrak{m}}^\ord$ at the kernel of $g$ is a discrete valuation ring. It follows that there is a unique minimal prime $\mathfrak{a}\subset\hecke_\mathfrak{m}^\ord$ such that $g$ factors through the integral domain
$$
R\define \hecke^\ord_\mathfrak{m}/\fa.
$$ 
If we let $\mathcal{L}$ and $\mathcal{K}$ denote the fraction fields of $\Lambda$ and $R$, respectively, then $\mathcal{K}$ is a finite extension of $\mathcal{L}$ and is the (primitive) component of  $\hecke_\mathfrak{m}^\ord\otimes_\Lambda\mathcal{L}$ to which $g$ \emph{belongs} in the sense of \cite[\S 1]{hida}. 
The $\Lambda$-algebra $\hecke^\ord_\mathfrak{m}$ is the  \emph{Hida family} of $g$ and $R$ is the \emph{branch} of the Hida family on which $g$ lives.   Define  $\hecke^\ord$-modules
\begin{eqnarray*}
\mathrm{Ta}_p^\ord(J_s)  &=&  e^\ord (\mathrm{Ta}_p(J_s)\otimes_{\Z_p}\co_F) \\
\mathbf{Ta}^\ord  &=& \mil \Ta_p^\ord(J_s) \\
 \mathbf{Ta}^\ord_\mathfrak{m} &=&  \mathbf{Ta}^\ord \otimes_{\hecke^\ord} \hecke_\mathfrak{m}^\ord \\
\FT &=&  \mathbf{Ta}^\ord_\mathfrak{m}\otimes_{\hecke^\ord_\mathfrak{m}} R
\end{eqnarray*}
(the Hecke operators $T_\ell$, $U_\ell$, and $\langle \ell\rangle$ act on $J_s$ and on the Tate module $\Ta_p(J_s)$ via the  Albanese action as in \cite[p. 236]{MW}, and the inverse limit in the second definition is with respect to $\alpha_*$). All four modules  admit  natural $\hecke^\ord$-linear actions of $G_\Q$.

\begin{Prop}\label{Gorenstein}
The $\hecke^\ord_\mathfrak{m}$-module $\mathbf{Ta}^\ord_\mathfrak{m}$ is free of rank two.   As a Galois representation $\mathbf{Ta}^\ord_\mathfrak{m}$ is unramified outside $Np$  and the arithmetic Frobenius of a  prime $\ell\nmid Np$ acts with characteristic polynomial $X^2-T_\ell X+[\ell] \ell$. Furthermore,  
$$
\hecke^\ord_\mathfrak{m}\iso \Hom_\Lambda(\hecke^\ord_\mathfrak{m},\Lambda)
$$ 
as $\hecke^\ord_\mathfrak{m}$-modules.
\end{Prop}

\begin{proof}
As we assume that the residual representation attached to  $\rho_g$ is irreducible, this is \cite[Th\'eor\`eme 7]{MT}.
\end{proof}

\begin{Def}\label{def:critical}
Factor the $p$-adic cyclotomic character $\epsilon_\cyc=\epsilon_\tame \cdot \epsilon_{\wild}$  as a product of characters taking values in $\mu_{p-1}$ and $1+p\Z_p$.  Define the \emph{critical character} $\Theta:G_\Q\map{}\Lambda^\times$ by
$$
\Theta=\epsilon_\tame^{\frac{k+j}{2}-1}\cdot [\epsilon_\wild^{1/2}]
$$
where $\epsilon_\wild^{1/2}$ is the unique square root of $\epsilon_\wild$ taking values in $1+p\Z_p$. 
Let $R^\dagger$ denote $R$ viewed as a module over itself but with  $G_\Q$ acting through $\Theta^{-1}$, and define the \emph{critical twist}  $\FT^\dagger=\FT\otimes_R R^\dagger$.  
\end{Def}

\begin{Rem}\label{twist remark}
The integer $j$ of the introduction is determined  only modulo $p-1$, and the resulting ambiguity in $\frac{k+j}{2}$ modulo $p-1$ means that $\Theta$ is determined only up to multiplication by the quadratic character  of conductor $p$.  The two possible choices of $\Theta$ arising from the two possible choices of $j$ modulo $2(p-1)$ are largely indistinguishable for our purposes, although the sign in the functional equation of the  two variable $p$-adic $L$-function of $R$, evaluated at $\Theta$ and viewed as a function on arithmetic primes, may depend on the choice.  See Remark \ref{functional ambiguity} and Proposition \ref{functional equation}.  We now fix a choice of $\Theta$ once and for all.
\end{Rem}

Using $\omega$ to identify $\Delta\iso\mu_{p-1}$, the idempotent $e_{k+j-2}\in\co_F[[\Z_p^\times]]$ satisfies  
$$
e_{k+j-2} \cdot [\zeta]=\zeta^{k+j-2} \cdot e_{k+j-2}
$$ 
for any $\zeta\in \mu_{p-1}$. As noted earlier $\hecke_\mathfrak{m}^\ord=e_{k+j-2}\hecke_\mathfrak{m}^\ord$, and so $[\epsilon_\tame]=\epsilon_\tame^{k+j-2}$ in $\hecke_\mathfrak{m}^\ord$.  It follows that 
\begin{equation}\label{square root}
\Theta^2(\sigma)  = [\epsilon_\cyc(\sigma)]
\end{equation}
in $\hecke_\mathfrak{m}^\ord$ for all $\sigma\in G_\Q$.  We will sometimes view $\Theta$ as a character $\Z_p^\times\map{}\Lambda^\times$ by factoring $\Theta$ through $\Gal(\Q(\mu_{p^\infty})/\Q)$ and using the isomorphism $$\epsilon_\cyc:  \Gal(\Q(\mu_{p^\infty})/\Q)\iso\Z_p^\times.$$  Then for all $z\in\Z_p^\times$, $\Theta^2(z)=[z]$ as elements of  $\hecke_\mathfrak{m}^\ord$.  Similarly
for any weight $r$ arithmetic prime $\mathfrak{p}$ of $R$, define $F_\mathfrak{p}^\times$-valued characters
$$
\Theta_\mathfrak{p}: G_\Q\map{\Theta}R^\times\map{}F_\mathfrak{p}^\times
\hspace{1cm}
[\cdot]_\mathfrak{p}: \Z_p^\times\map{[\cdot]}R^\times\map{}F_\mathfrak{p}^\times.
$$
As above we may view $\Theta_\mathfrak{p}$ as a character of $\Z_p^\times$, so that 
$$
\Theta_\mathfrak{p}(\delta\gamma)  =  \omega^{\frac{k+j}{2}-1}(\delta) \cdot \psi_\mathfrak{p}^{1/2}(\gamma) \cdot  \gamma^{\frac{r-2}{2}} 
$$
for all $\delta\in\Delta$ and $\gamma\in\Gamma$. 

By  \cite[\S 1.6.10]{nek} or \cite[\S 4]{ohta} and (\ref{square root})  there is a perfect, alternating, $G_\Q$-invariant, $\Lambda$-bilinear pairing 
$$
\mathbf{Ta}^\ord_\mathfrak{m}\times \mathbf{Ta}^\ord_\mathfrak{m}  \map{}\Lambda(1)\otimes\Theta^2
$$ 
where $\Lambda(1)$ denotes the usual Tate twist of $\Lambda$ (and the unadorned $\Lambda$ has trivial Galois action). By the discussion of \cite[\S 1.6.10]{nek} and the final claim of Proposition \ref{Gorenstein}, this pairing induces a perfect $R$-bilinear pairing 
\begin{equation}\label{pairing}
\FT^\dagger \times\FT^\dagger \map{}R(1).
\end{equation}

 Given an arithmetic prime  $\mathfrak{p}\subset R$ of weight $r$ set $s=\max\{1,\ord_p(\mathrm{cond}(\psi_\mathfrak{p})) \}$. By a fundamental result of Hida \cite[Theorem 1.2]{hida} the composition 
$$
\hecke^\ord\map{}\hecke_\mathfrak{m}^\ord\map{}R\map{}F_\mathfrak{p} 
$$  
factors through $ \hecke_{r,s}^\ord$ and determines an ordinary $p$-stabilized newform  
\begin{eqnarray*}\label{special form}
g_\mathfrak{p}\in S_r(\Phi_s,  \psi_\mathfrak{p}\omega^{k+j-r} , F_\mathfrak{p}).
\end{eqnarray*} 
We define a Galois representation
$$
V_\mathfrak{p}^\dagger=\FT^\dagger\otimes_R F_\mathfrak{p} \iso \FT^\dagger_\mathfrak{p}/ \mathfrak{p}\FT^\dagger_\mathfrak{p}.
$$
Tensoring the pairing (\ref{pairing})  with $F_\mathfrak{p}$ yields an alternating nondegenerate pairing 
$$
V_\mathfrak{p}^\dagger\times V_\mathfrak{p}^\dagger\map{}F_\mathfrak{p}(1).
$$

\begin{Lem}\label{oldform}
Suppose $\mathfrak{p}$ is an arithmetic prime of $R$ of weight $r>2$ and trivial character (so that $r$ must be even).  Then the form $g_\mathfrak{p}\in S_r(\Gamma_0(Np),F_\mathfrak{p})$ is old at $p$.
\end{Lem}

\begin{proof}
If $g_\mathfrak{p}$ were new at $p$ then $U_p$ would act with eigenvalue $\alpha_\mathfrak{p}$ satisfying $\alpha_\mathfrak{p}^2=p^{r-2}$ \cite[Theorem 4.6.17]{miyake}, contradicting $g_\mathfrak{p}$ being ordinary. 
\end{proof}

\begin{Lem}\label{DVR}
For any arithmetic prime $\mathfrak{p}\subset R$ the localization $R_\mathfrak{p}$ is a discrete valuation ring.
\end{Lem}

\begin{proof}
This follows from the discussion of  \cite[\S 12.7.5]{selmer_complexes}.
\end{proof}

\begin{Lem}\label{ring lemma}
Suppose $M$ is a finitely generated $R$-module and $m\in M$ is nontorsion. Then $m\not\in \mathfrak{p}M_\mathfrak{p}$ for all but finitely many arithmetic primes $\mathfrak{p}\subset R$.
\end{Lem}

\begin{proof}
Let $I\subset R$ be the image of the map $\Hom_{R}(M,R)\map{}R$  defined by $f\mapsto f(m)$. For any arithmetic prime $\mathfrak{p}$
 $$
 m\in\mathfrak{p}M_\mathfrak{p} \implies I\subset \mathfrak{p}R_\mathfrak{p}
 \implies (R/I)_\mathfrak{p}\not=0.
 $$
By   \cite[Theorem 6.5]{Mat} this can only occur for finitely many $\mathfrak{p}$.
\end{proof}

%%%%%%%%%%%%%%%%%%%%%%%%%%%%%%%%%%%%%%%%%%%%%%%%%%%%%%%%%%%%%%%%%%%%%

\subsection{Construction of big Heegner points}
\label{ss:construction}

%%%%%%%%%%%%%%%%%%%%%%%%%%%%%%%%%%%%%%%%%%%%%%%%%%%%%%%%%%%%%%%%%%%%

Fix a quadratic imaginary field $K$ with maximal order $\co$ as in the introduction, and an ideal $\mathfrak{N}$ with $\co/\mathfrak{N}\iso \Z/N\Z$. For any positive integer $c$ we denote by $\co_c$  the order of conductor $c$ in $K$ and by $H_c$ the ring class field of $\co_c$.  Let $H_c^{(Np)}$ denote the maximal extension of $H_c$ unramified outside $Np$ and set 
$$
\mathfrak{G}_c=\Gal(H_c^{(Np)}/H_c).
$$
Let $\widehat{\Q}$ and $\widehat{K}$ denote the rings of finite adeles of $\Q$ and $K$, respectively. 
 
Now fix a positive integer $c$ prime to $N$.  For each  integer $s\ge 0$ define an elliptic curve $E_{c,s}$ over $\C$ with complex multiplication by $\co_{cp^s}$ by
$$
E_{c,s}(\C)\iso \C/\co_{cp^s}.
$$
The $\mathfrak{N}\cap\co_{cp^s}$-torsion subgroup 
$$
\fn_{c,s}=E_{c,s}[\mathfrak{N}\cap\co_{cp^s}]
$$ 
is  cyclic of order $N$.  The inclusion $\co_{cp^{s+1}}\subset \co_{cp^s}$ induces a $p$-isogeny $E_{c,s+1}\map{}E_{c,s}$ compatible with the action of $\co_{cp^{s+1}}$ on the source and target, and taking  $\fn_{c,s+1}$ isomorphically to $\fn_{c,s}$.   The kernel of the composition
$$
j_{c,s}: E_{c,s}\map{}E_{c,s-1}\map{}\cdots \map{} E_{c,1}\map{}E_{c,0}
$$
is cyclic of order $p^s$ and is characterized as the $p^s\co_c$-torsion in $E_{c,s}$.  Fix a generator $\varpi$ of $\co/\Z\iso\Z$. Then $c \varpi$ generates 
$$
\ker(j_{c,s})\iso \co_c/\co_{cp^s}
$$ 
for every $s$, and defines a $\Gamma_1(p^s)$-level structure $\pi_{c,s}=c\varpi\in E_{c,s}[p^s]$.  Define
\begin{equation}\label{first point}
h_{c,s}= (E_{c,s},\fn_{c,s},\pi_{c,s})\in X_s(\C).
\end{equation}
%We will see momentarily that $h_{c,s}$ is rational over $H_{cp^s}(\mu_{p^s})$.

By class field theory  $\Q(\sqrt{p^*})\subset H_p$, where $p^*=(-1)^{\frac{p-1}{2}}p$.  The restriction of  $\epsilon_\cyc$ to $\Gal(\overline{\Q}/\Q(\sqrt{p^*}))$ takes values in $(\Z_p^\times)^2$, and it follows that there is a unique continuous homomorphism 
$$
\vartheta:\Gal(\overline{\Q}/\Q(\sqrt{p^*})) \map{}\Z_p^\times/\{\pm 1\}
$$
such that $\vartheta^2=\epsilon_\cyc$.

\begin{Lem}\label{yoke}
Suppose $s>0$. The field of moduli of  $E_{c,s}$ is contained in  $H_{cp^s}$.  After fixing a model of $E_{c,s}$ over $H_{cp^s}$ the subgroup $\fn_{c,s}$ is also defined over $H_{cp^s}$, and for each $\sigma\in \Gal(\overline{\Q}/H_{cp^s})$ and $P\in \ker(j_{c,s})$ the equality  
$
P^\sigma =\vartheta(\sigma) P
$  
holds up to the action of $\Aut_{\overline{\Q}}(E_{c,s})=\{\pm 1\}$.
\end{Lem}

\begin{proof}
Fix $\sigma\in\Aut(\C/H_{cp^s})$ and let $x\in \widehat{K}^\times$ be a finite idele whose (arithmetic) Artin symbol is equal to the restriction of $\sigma$ to the maximal abelian extension of $K$.  Let $\widehat{\co}_{cp^s}$ be the closure of $\co_{cp^s}$ in $\widehat{K}$.  As $\sigma$ fixes $H_{cp^s}$ we must have $x\in K^\times\cdot \widehat{\co}_{cp^s}^\times$,  and multiplying $x$ by an element of $K^\times$ we may assume $x\in\widehat{\co}_{cp^s}^\times$. The main theorem of complex multiplication \cite[Theorem 5.4]{shimura} then gives an isomorphism of complex tori
$$
E_{c,s}(\C)\iso \C/\co_{cp^s}= \C/x^{-1}\co_{cp^s} \iso E^\sigma_{c,s}(\C).
$$
Thus $E_{c,s}$ has a model over $H_{cp^s}$, which we now fix.  All endomorphisms of $E_{c,s}$ are then defined over $H_{cp^s}$ by \cite[(5.1.3)]{shimura}, and the characterizations of  $\fn_{c,s}$ and $\ker(j_{c,s})$ in terms of $I$-torsion subgroups for ideals $I\subset\co_{cp^s}$ shows that they are also defined over $H_{cp^s}$. Let $x_p\in (\co_{cp^s}\otimes\Z_p)^\times$ be the $p$-component of $x$, and write $x_p=\alpha+cp^s \beta$ with $\alpha\in\Z_p^\times$ and $\beta\in \co\otimes\Z_p$. In particular 
$$
\alpha^2\equiv \mathrm{N}_{K/\Q}(x_p)\pmod{p^s}.
$$
Again applying the main theorem of complex multiplication we obtain a commutative diagram
$$
\xymatrix{
{\Z/p^s\Z} \ar[r]\ar[d]_{\alpha^{-1}\cdot } &
{(\co_{c}\otimes\Z_p)/(\co_{cp^s}\otimes\Z_p) }
\ar[rr]^\xi \ar[d]_{x_p^{-1}\cdot}    & & {\ker(j_{c,s})} \ar[d]_\sigma \\
{\Z/p^s\Z} \ar[r]  &
{(\co_{c}\otimes\Z_p)/(\co_{cp^s}\otimes\Z_p) } \ar[rr]^{\xi'}
&  & {\ker(j_{c,s})}
}
$$
(the unlabeled horizontal arrows are the isomorphisms determined by $1\mapsto c\varpi$) for some group isomorphisms $\xi$ and $\xi'$ which agree up to the action  of $\co_{cp^s}^\times=\{\pm 1\}$. As the composition
$$
\Z_p^\times\map{}\widehat{\Q}^\times \map{}\Gal(\Q(\mu_{p^\infty})/\Q) \map{\epsilon_\cyc} \Z_p^\times
$$
is given by $y\mapsto y^{-1}$ and takes $\mathrm{N}_{K/\Q}(x_p)$ to $\epsilon_\cyc(\sigma)$ (the first arrow is the natural inclusion, the second is the arithmetic  Artin symbol), we see that 
$$
\epsilon_\cyc(\sigma) = \mathrm{N}_{K/\Q}(x_p^{-1})\equiv \alpha^{-2}\pmod{p^s}.
$$
Hence the action of $\sigma$ on $\ker(j_{c,s})$ is given as multiplication by the unique (up to $\pm 1$)  square root of  $\epsilon_\cyc(\sigma)$ in $(\Z/p^s\Z)^\times$.
\end{proof}

\begin{Cor}\label{yoke lemma}
Suppose $s>0$ and let $L_{c,s}=H_{cp^s}(\mu_{p^s})$.  Then $h_{c,s}\in X_s(L_{c,s})$ and
\begin{equation}\label{yoke lemma display}
h_{c,s}^\sigma=\langle\vartheta(\sigma)\rangle\cdot h_{c,s} 
\end{equation}
 for all $\sigma\in \Gal(L_{c,s}/H_{cp^s})$.
 \end{Cor}

\begin{proof}
Indeed, Lemma \ref{yoke} asserts that (\ref{yoke lemma display}) holds for all $\sigma\in\Gal(\overline{\Q}/H_{cp^s})$.  Suppose $\sigma$ fixes $L_{c,s}$.  Then
$$
\epsilon_\cyc(\sigma)\equiv 1\pmod{p^s}
\implies
\vartheta(\sigma)\equiv \pm 1\pmod{p^s},
$$ 
and so $\langle \vartheta(\sigma)\rangle$ acts trivially on $X_s$.
\end{proof}

  As in \cite[\S 7.2 Lemma 1]{hida-book} we may define the ordinary projector $e^\ord=\lim U_p^{m!}$ acting on the Picard group $\mathrm{Pic}(X_{s /L_{c,s}})\otimes \co_F$. The natural short exact sequence
$$
0\map{}J_s(L_{c,s})\otimes \co_F  \map{}  \mathrm{Pic}( X_{s /L_{c,s}})\otimes \co_F  \map{\deg }\co_F\map{}0,
$$
is Hecke equivariant, and as the action of $U_p$ on  $\co_F$ is by $p=\deg(U_p)$ there is an induced isomorphism
$$
J_s(L_{c,s})^\ord\iso \mathrm{Pic}( X_{s /L_{c,s}})^\ord.
$$
Here we abbreviate $$J_s(L)^\ord=e^\ord(J_s(L)\otimes\co_F)$$ for any finite extension $L/\Q$, and similarly for the Picard group.   Viewing $h_{c,s}$ as divisor on $X_{s /L_{c,s}}$ we obtain an element 
$$e^\ord h_{c,s}\in J_s(L_{c,s})^\ord.$$

If $\sigma\in \Gal(\overline{\Q}/H_{cp^s})$ then $\sigma$ fixes $\Q(\sqrt{p^*})$, and hence there is a $\xi\in \mu_{p-1}$ for which $\xi^2=\epsilon_\tame(\sigma)$.  This implies $\xi\cdot \epsilon_\wild^{1/2}(\sigma)=\pm \vartheta(\sigma)$, and so 
$$
\Theta(\sigma)=\xi^{ k+j-2}  \langle   \epsilon^{1/2}_\wild(\sigma) \rangle=
\langle \xi \cdot \epsilon_\wild^{1/2} (\sigma)\rangle = \langle\vartheta(\sigma)\rangle
$$
as endomorphisms of  $e_{k+j-2} J_s(L_{c,s})^\ord$.  If we define 
$$
y_{c,s}=e_{k+j-2}e^\ord h_{c,s}\in J_s(L_{c,s})^\ord
$$ 
for every $s>0$ then Corollary \ref{yoke lemma} implies that 
\begin{equation}\label{jacobian twist}
y_{c,s}^\sigma=\Theta(\sigma) \cdot y_{c,s}
\end{equation}
for all $\sigma\in \Gal(\overline{\Q}/H_{cp^s})$.
Let $\hecke_{2,s}^{\ord,\dagger}$ denote  $\hecke_{2,s}^\ord$ as a module over itself but with $G_\Q$ acting through the character $ \Theta^{-1}$, and   let $\zeta_s\in \hecke_{2,s}^{\ord,\dagger}$ be the element corresponding to $1\in \hecke_{2,s}^\ord$ under the identification of underlying $\hecke_{2,s}^\ord$-modules. For any $\hecke_{2,s}^\ord$-module $M$ we abbreviate
$$
M\otimes\zeta_s=M\otimes_{\hecke_{2,s}^\ord} \hecke_{2,s}^{\ord,\dagger}.
$$
The equality (\ref{jacobian twist})  implies that
$$
y_{c,s}\otimes\zeta_s\in H^0(H_{cp^s}, J_s(L_{c,s})^\ord\otimes\zeta_s)
$$
and we define
$$
x_{c,s}=\mathrm{Cor}_{H_{cp^s}/H_c} (y_{c,s}\otimes\zeta_s) \in H^0(H_{c}, J_s(L_{c,s})^\ord\otimes\zeta_s)
$$
where $\mathrm{Cor}_{H_{cp^s}/H_c}$ is corestriction. More explicitly, if for each $\eta\in\Gal(H_{cp^s}/H_c)$ we fix an extension to $\Gal(L_{c,s}/H_c)$ then
\begin{equation}\label{medium heegner}
x_{c,s}=\left(\sum_{\eta\in \Gal(H_{cp^s}/H_c)} \Theta(\eta^{-1}) \cdot  y_{c,s}^\eta\right) \otimes\zeta_s\in J_s(L_{c,s})^\ord\otimes\zeta_s.
\end{equation}

We now construct a twisted Kummer map
$$
\mathrm{Kum}_s: H^0(H_{c}, J_s(L_{c,s})^\ord\otimes\zeta_s)
\map{}
H^1(\mathfrak{G}_c,  \mathrm{Ta}_p^\ord(J_s) \otimes\zeta_s).
$$
  Suppose $P\otimes\zeta_s\in J_s(L_{c,s})^\ord\otimes\zeta_s$ is fixed by the action of $\Gal(\overline{\Q}/H_{c})$.  For each $n>0$ choose a finite extension $L/L_{c,s}$ contained in $H_c^{(Np)}$ large enough so that there is a point $Q_n \in J_s(L)^\ord$ with $p^nQ_n=P$.  Abbreviating
  $$
  J_s[p^n]^\ord= e^\ord(J_s[p^n]\otimes\co_F),
  $$
for  $\sigma\in \mathfrak{G}_c$
\begin{eqnarray*}
b_n(\sigma)&=&(Q_n \otimes\zeta_s)^\sigma-Q_n\otimes\zeta_s\\
&=& \big(\Theta^{-1}(\sigma)Q_n^\sigma-Q_n\big) \otimes\zeta_s
\end{eqnarray*}
defines a $1$-cocycle with values in 
$$
J_s[p^n]^\ord \otimes\zeta_s\iso (\mathrm{Ta}_p^\ord(J_s)/p^n\mathrm{Ta}_p^\ord(J_s)) \otimes\zeta_s
$$
whose image in cohomology  does not depend on the choice of $L$ or $Q_n$.  Taking the inverse limit over $n$ yields the desired element 
$$
\mathrm{Kum}_s(P\otimes\zeta_s)=\mil b_n \in H^1(\mathfrak{G}_c,  \mathrm{Ta}_p^\ord(J_s) \otimes\zeta_s).
$$  
The twisted Kummer map is both $\hecke_{2,s}^\ord$ and $G_\Q$ equivariant.   Define
\begin{equation}\label{twisted norm}
\mathfrak{X}_{c,s}= \mathrm{Kum}_s(x_{c,s}) \in H^1(\mathfrak{G}_c,  \mathrm{Ta}_p^\ord(J_s) \otimes\zeta_s).
\end{equation}

The Albanese map $\alpha_*: J_{s+1}  \map{} J_s  $ induces a map of $\hecke^\ord$-modules (abusively denoted the same way)
$$
\alpha_*: \mathrm{Ta}_p^\ord(J_{s+1}) \otimes\zeta_{s+1}\map{} \mathrm{Ta}_p^\ord(J_s) \otimes\zeta_s
$$  
defined by $t\otimes \zeta_{s+1}\mapsto \alpha_*(t)\otimes\zeta_s$.  Taking the inverse limit with respect to $\alpha_*$ we obtain an $\hecke^\ord$-module
$$
\mathbf{Ta}^\ord\otimes\zeta  =  \mil \left(\mathrm{Ta}_p^\ord(J_s) \otimes\zeta_s\right)
$$
which is precisely the module $\mathbf{Ta}^\ord$ constructed earlier, but with the $G_\Q$ action twisted by $\Theta^{-1}$.  In particular there is a $G_\Q$ equivariant map of $\hecke^\ord$-modules
\begin{equation}\label{twisted quotient}
\mathbf{Ta}^\ord\otimes\zeta\map{}\FT^\dagger
\end{equation}

\begin{Def}\label{Def:big}
By Lemma \ref{Euler 0} below we may take the limit over $s$ of the  $\mathfrak{X}_{c,s}$ and form the cohomology class
\begin{equation}\label{heegner limit}
\mil U_p^{-s}\mathfrak{X}_{c,s}\in H^1(\mathfrak{G}_c,\mathbf{Ta}^\ord\otimes\zeta).
\end{equation}
Define the  \emph{big Heegner point of conductor $c$}
$$
\mathfrak{X}_c\in H^1(\mathfrak{G}_c,\FT^\dagger)
$$
to be the image of the cohomology class (\ref{heegner limit}) under the map on cohomology induced by (\ref{twisted quotient}).
\end{Def}

\begin{Lem}\label{Euler 0}
The cohomology classes (\ref{twisted norm})  satisfy
 $
 \alpha_*(\mathfrak{X}_{c,s+1})=U_p\cdot \mathfrak{X}_{c,s}.
 $
\end{Lem}

\begin{proof}
The extensions $H_{cp^s}(\mu_{p^\infty})$  and $H_{cp^{s+1}}$ of $H_{cp^s}$ are linearly disjoint, hence we may fix a set $S\subset \Aut(\C/H_{cp^s})$ of extensions of $\Gal(H_{cp^{s+1}}/H_{cp^s})$ in such a way that each $\sigma\in S$ acts trivially on $\mu_{p^\infty}$. 
The map of complex tori 
$$
\C/\co_{cp^s}\map{}\C/\co_{cp^{s+1}}
$$ 
defined by $z\mapsto pz$ determines a $p$-isogeny  $f_s:E_{c,s}\map{}E_{c,s+1}$ taking $\fn_{c,s}$ to $\fn_{c,s+1}$  and $\pi_{c,s}$ to $p\cdot \pi_{c,s+1}$.  That is, $f_s$ determines  an isogeny $ f_s: h_{c,s}\map{}\alpha(h_{c,s+1})$  of elliptic curves over $\C$ with $\Phi_s$ level structure.  By Corollary \ref{yoke lemma} each $\sigma\in S$ fixes $h_{c,s}$,  while the main theorem of complex multiplication implies that the complex tori $E_{c,s+1}^\sigma(\C)$  as $\sigma$ ranges over $S$ are isomorphic to $\C/\mathfrak{b}$ as $\mathfrak{b}$ ranges over a set of representatives for the kernel of $\mathrm{Pic}(\co_{cp^{s+1}})\map{}\mathrm{Pic}(\co_{cp^s})$.   The set
$$
\{\C/\mathfrak{b}\mid \mathfrak{b}\in \mathrm{ker}(\mathrm{Pic}(\co_{cp^{s+1}})\map{}\mathrm{Pic}(\co_{cp^s}))\}
$$
consists of $p$ nonisomorphic complex elliptic curves, and so as $\sigma$ varies over $S$ the elliptic curves  $E_{c,s+1}^\sigma$ are pairwise nonisomorphic. Hence the  isogenies  
$$
\{ f^\sigma_s:E_{c,s}\map{}\alpha(E_{c,s+1}^\sigma)   \mid \sigma\in S\}
$$ 
have distinct kernels and we obtain $p$ distinct degree $p$ isogenies of elliptic curves with $\Phi_s$ level structure
$$
\{ f^\sigma_s:h_{c,s}\map{}\alpha(h_{c,s+1}^\sigma)   \mid \sigma\in S\}.
$$ 
   By definition of the $U_p$ correspondence
\begin{equation}\label{Euler I equ 2}
U_p\cdot h_{c,s}= \sum_{\sigma\in S}\alpha (h_{c,s+1}^\sigma) 
\end{equation}
as divisors on $X_{s/L_{c,s}}$.  Applying $e_{k+j-2}\cdot e^\ord$ we obtain the equality in $J_s(L_{c,s})^\ord$
$$
U_p \cdot y_{c,s}=\sum_{\sigma\in S}\alpha_* (y_{c,s+1}^\sigma).
$$
Twisting by $\zeta_s$ we obtain the equality in $H^0(H_{cp^s}, J_s(L_{c,s})^\ord\otimes\zeta_s)$
\begin{eqnarray}\label{p-part Euler display}
U_p (y_{c,s} \otimes\zeta_s )&=&
 \sum_{\sigma\in S}\alpha_* (y_{c,s+1}^\sigma) \otimes\zeta_s \\
 %&=&
 %\sum_{\sigma\in S}\alpha (e^\ord h_{c,s+1}^\sigma \otimes\zeta_{s+1}) \\
&=&\nonumber
  \sum_{\sigma\in S}\alpha_* (y_{c,s+1} \otimes\zeta_{s+1})^\sigma \\
  &=&\nonumber
  \alpha_*(\mathrm{Cor}_{H_{cp^{s+1}}/H_{cp^s}}  (y_{c,s+1} \otimes\zeta_{s+1})).
\end{eqnarray}
Corestricting from $H_{cp^s}$ to $H_c$ shows that
$$
U_p \cdot x_{c,s} = \alpha_*(x_{c,s+1})
$$
and applying the twisted Kummer map to both sides gives
$$
U_p\cdot\mathfrak{X}_{c,s} = \mathrm{Kum}_s(\alpha_*(x_{c,s+1}))
=\alpha_* (\mathrm{Kum}_{s+1}  (x_{c,s+1}))=\alpha_*(\mathfrak{X}_{c,s+1}).
$$
\end{proof}

%%%%%%%%%%%%%%%%%%%%%%%%%%%%%%%%%%%%%%%%%%%%%%%%%%%%%%%%%%%%%%%%%%%%

\subsection{Euler system relations}
\label{ss:esr}

%%%%%%%%%%%%%%%%%%%%%%%%%%%%%%%%%%%%%%%%%%%%%%%%%%%%%%%%%%%%%%%%%%%%

Keep the notation of \S \ref{ss:construction}.  In particular  $c$ always denotes a positive integer prime to $N$  and $\mathfrak{X}_c\in H^1(H_c,\FT^\dagger)$ is the inflation to $\Gal(\overline{\Q}/H_c)$-cohomology of the big Heegner point of Definition \ref{Def:big}.

\begin{Prop}\label{Prop:ESrelations}
Corestriction from $H_{cp}$ to $H_c$ takes $$\mathfrak{X}_{cp}\mapsto U_p\cdot \mathfrak{X}_{c}.$$ For any prime  $\ell\nmid c N$ which is inert in $K$, corestriction from $H_{c\ell}$ to $H_c$ takes $$\mathfrak{X}_{c\ell}\mapsto T_\ell\cdot  \mathfrak{X}_c.$$
\end{Prop}

\begin{proof}
Directly from the definition (\ref{first point}) we have the equality $h_{cp,s}=\alpha(h_{c,s+1})$ of points on $X_s(\C)$.  Tracing through the constructions of \S \ref{ss:construction}, and using (\ref{p-part Euler display}) for the third implication, gives
\begin{eqnarray*}\lefteqn{
y_{cp,s} = \alpha_*(y_{c,s+1})} \\
&\implies &\ y_{cp,s}\otimes\zeta_s =\alpha_*(y_{c,s+1}\otimes\zeta_{s+1}) \\
&\implies& \ \mathrm{Cor}_{H_{cp^{s+1}}/H_{cp^s}} (y_{cp,s}\otimes\zeta_s)
= \alpha_*( \mathrm{Cor}_{H_{cp^{s+1}}/H_{cp^s}} (y_{c,s+1}\otimes\zeta_{s+1}) ) \\
&\implies& \ 
\mathrm{Cor}_{H_{cp^{s+1}}/H_{cp^s}} (y_{cp,s}\otimes\zeta_s)
= U_p\cdot  (y_{c,s}\otimes\zeta_s) \\
&\implies& \ 
\mathrm{Cor}_{H_{cp^{s+1}}/H_{c}} (y_{cp,s}\otimes\zeta_s)
= U_p\cdot \mathrm{Cor}_{H_{cp^s}/H_c} (y_{c,s}\otimes\zeta_s) \\
&\implies& \ \mathrm{Cor}_{H_{cp}/H_c} (x_{cp,s})=U_p\cdot x_{c,s}.
\end{eqnarray*}
As the twisted Kummer map is both Hecke and Galois equivariant we obtain 
$$
\mathrm{Cor}_{H_{cp}/H_c} (\mathfrak{X}_{cp,s})=U_p\cdot \mathfrak{X}_{c,s},
$$
and passing to the limit proves the first claim of the proposition.

Fix a prime $\ell$ as in the second claim of the proposition.
The map $z\mapsto \ell z$ induces an $\ell$-isogeny  $\C/\co_{cp^s}\map{}\C/\co_{c\ell p^s}$ of elliptic curves over $\C$, which determines an isogeny of elliptic curves with $\Phi_s$ level structure  
$$
f: h_{c,s}\map{}h_{c\ell,s}.
$$  
That is, $f$ takes the extra $\Phi_s$ level structure on $E_{c,s}$ isomorphically to that on $E_{c\ell,s}$.
 As in the proof of Lemma \ref{Euler 0} we may fix a subset $S\subset\Aut(\C/H_{cp^s})$ of extensions of  $\Gal(H_{c\ell p^s}/H_{cp^s})$ in such a way that each $\sigma\in S$  acts trivially on $\mu_{p^\infty}$. As $\sigma$ ranges over $S$ the complex tori $E^\sigma_{c\ell ,s}(\C)$ are given by $\C/\mathfrak{b}$ as $\mathfrak{b}$ ranges over representatives of the kernel of  $\mathrm{Pic}(\co_{c\ell p^s})\map{}\mathrm{Pic}(\co_{c p^s})$. The $\ell+1$ lattices $\mathfrak{b}$ which arise in this way are not $\C^\times$-homothetic, and therefore the $S$-conjugates of $E_{c\ell,s}$ are pairwise nonisomorphic.
 As $\sigma\in S$ varies we obtain $\ell+1$ distinct $\ell$-isogenies $f^\sigma:h_{c,s}\map{}h_{c\ell,s}^\sigma$ of elliptic curves with $\Phi_s$ level structure, giving the equality of divisors 
 \begin{equation}\label{pre-congruence}
  T_\ell\cdot h_{c,s} = \sum_{\sigma\in S} h_{c\ell,s}^\sigma.
 \end{equation}
From this we deduce
\begin{eqnarray*}\lefteqn{
   \sum_{\sigma\in S} y_{c\ell,s}^\sigma = T_\ell\cdot y_{c,s} }  \\
 &\implies&\  \sum_{\sigma\in S} (y_{c\ell,s}\otimes\zeta_s)^\sigma = T_\ell\cdot (y_{c,s} \otimes\zeta_s )\\
&\implies& \ \mathrm{Cor}_{H_{c\ell p^s}/H_{cp^s}} (y_{c\ell,s}\otimes\zeta_s) =  T_\ell\cdot (y_{c,s}\otimes\zeta_s)  \\
&\implies& \ 
\mathrm{Cor}_{H_{c\ell p^s} /H_c} (y_{c\ell ,s}\otimes\zeta_s)
= T_\ell \cdot  (x_{c,s}\otimes\zeta_s) \\
&\implies& \ 
\mathrm{Cor}_{ H_{c\ell} /H_c} (x_{c\ell ,s})
= T_\ell \cdot  x_{c,s} \\
&\implies& \ 
\mathrm{Cor}_{  H_{c\ell} /H_c} (\mathfrak{X}_{c\ell ,s})
= T_\ell \cdot  (\mathfrak{X}_{c,s}).
\end{eqnarray*}
Passing to the limit in $s$ proves the claim.
\end{proof}

\begin{Prop}\label{congruence}
Suppose $c$ is positive integer  which is prime to $N$,  and that $\ell\nmid cN$ is a prime which is inert in $K$.   Fix a prime $v$ of  $H_c$ above $\ell$ and let $w$ be a place of $\overline{\Q}$  above $v$.  If $\Frob_v\in \Gal(H_c/\Q)$ is the arithmetic Frobenius of $v$ then $\mathfrak{X}_{c\ell}$ and $\Frob_v(\mathfrak{X}_c)$ have the same image in $H^1(H_{c\ell,w},\FT^\dagger)$.
\end{Prop}

\begin{proof}
Fix $s>0$ and recall $L_{c,s}=H_{cp^s}(\mu_{p^s})$. Set
$$
\Phi_0=(L_{c,s})_w  \hspace{1cm} \Phi=(L_{c\ell,s})_w
$$
 so that $\Phi/\Phi_0$ is totally ramified of degree $\ell+1$ and 
 $$
 \Gal(\Phi/\Phi_0)\iso\Gal(H_{c\ell p^s}/H_{cp^s}).
 $$   
 As in the proof of Proposition \ref{Prop:ESrelations}, fix a set $S\subset\Aut(\C/H_{cp^s})$ of representatives for $\Gal(H_{c\ell p^s}/H_{cp^s})$ in such a way  that each $\sigma\in S$ acts trivially on $\mu_{p^\infty}$, so that the elements of $S$ give a set of representatives for $\Gal(\Phi/\Phi_0)$.   Let $W$ denote the integer ring of $\Phi$, $\mathbb{L}$ the residue field of $\Phi$, and let $\underline{X}_s$ be the canonical (smooth, proper) integral model of $X_{s/\Phi}$ over $W$.  By the valuative criterion of properness any point $ x \in X_s(\Phi)$ extends to a point denoted $\underline{x}\in  \underline{X}_s(W)$.  
 
 Fix $\tau\in\Aut(\C/H_c)$.  The equality (\ref{pre-congruence}) gives
 $$
\big( \sum_{\sigma\in S} h^\sigma_{c\ell,s} \big)^\tau =T_\ell\cdot h^\tau_{c,s}
 $$
as divisors on $X_{s/L_{c\ell,s}}$. First replacing $S$ by $\tau  S\tau^{-1}$ and then changing base to $\Spec(\Phi)$ gives the equality of divisors
$$
\sum_{\sigma\in \Gal(\Phi/\Phi_0)} (h^\tau_{c\ell,s})^\sigma = T_\ell\cdot h^\tau_{c,s}
$$
on $X_{s/\Phi}$.  As $\Phi/\Phi_0$ is totally ramified this implies the equality of divisors in the special fiber $\underline{X}_{s/\mathbb{L}}$
 $$
(\ell+1)\cdot  \underline{h}^\tau_{c\ell,s/ \mathbb{L} }  =
T_\ell\cdot \underline{h}^\tau_{c,s/\mathbb{L}}
 $$
where we abbreviate
$$
\underline{h}^\tau_{c,s/  \mathbb{L}}= \underline{h^\tau_{c,s}}\times_{\Spec(W)}\Spec(\mathbb{L})
$$ 
and similarly with $c$ replaced by $c\ell$.  Using the Eichler-Shimura relation we deduce
$$
(\ell+1)\cdot  \underline{h}^\tau_{c\ell,s/\mathbb{L}} =
\Frob_v (\underline{h}^\tau_{c,s/\mathbb{L}})+ \langle\ell\rangle \ell\cdot  \Frob_v^{-1} (\underline{h}^\tau_{c,s/\mathbb{L}}).
$$
The elliptic curve underlying   $\underline{h}_{c,s}$ has supersingular reduction, and the action of $\Frob_v^2$ on supersingular points in the special fiber is by  $\langle \ell\rangle$ (see the proof of \cite[Lemma 4.1.1]{howard-cycles}). Therefore
$$
(\ell+1)\cdot  \underline{h}^\tau_{c\ell,s/\mathbb{L}} =
(\ell+1)  \cdot  \Frob_v (\underline{h}^\tau_{c,s/\mathbb{L}})
$$
and so $\underline{h}^\tau_{c\ell,s}$ and $\Frob_v (\underline{h}^\tau_{c,s})$ have the same reduction to the special fiber.  
Let $\underline{J}_s$ denote the N\'eron model of $J_s$ over $W$ and abbreviate
$$
\underline{J}_s(\mathbb{L})^\ord=e^\ord(\underline{J}_s(\mathbb{L})\otimes\co_F).
$$
As the reduction map $J_s(\Phi)\map{}\underline{J}_s(\mathbb{L})$ is Hecke equivariant, we deduce from the discussion above that $y^\tau_{c\ell,s}$ and $\Frob_v(y^\tau_{c,s})$ have the same image under
$$
J_s(L_{c\ell,s})^\ord\map{} \underline{J}_s(\mathbb{L})^\ord.
$$
Letting $\tau$ vary over a set of representatives for $\Gal(H_{c\ell p^s}/H_{c\ell})$, which also gives representatives for  $\Gal(H_{cp^s}/H_{c})$, the definition (\ref{medium heegner}) shows that  $x_{c\ell,s}$ and $\Frob_v(x_{c,s})$ have the same image under
\begin{equation}\label{jacobian reduction}
J_s(L_{c\ell,s})^\ord\otimes\zeta_s \map{} \underline{J}_s(\mathbb{L})^\ord\otimes\zeta_s.
\end{equation}

Let $\mathbb{F}$ denote the residue field of $H_{c\ell,w}$.  The next claim is that the kernel of
\begin{eqnarray}\lefteqn{ \label{local kummer}
H^0(H_{c\ell},  J_s(L_{c,s})^\ord \otimes\zeta_s)\map{\mathrm{Kum}_s}
H^1(H_{c\ell}, \Ta_p^\ord(J_s)\otimes\zeta_s) }  \hspace{2cm}   \\
& & \map{\mathrm{loc}_w}H^1(H_{c\ell,w},\Ta_p^\ord(J_s)\otimes\zeta_s)
\nonumber
\end{eqnarray}
contains the kernel of the map
\begin{equation}\label{residual kummer}
H^0(H_{c\ell},  J_s(L_{c,s})^\ord \otimes\zeta_s) \map{}H^0(\mathbb{F},  J_s(\mathbb{L})^\ord \otimes\zeta_s)
\end{equation}
induced by the reduction map (\ref{jacobian reduction}).   Indeed, suppose $P\otimes\zeta_s$ is in the kernel of (\ref{residual kummer}), and let $Q_n$ and $b_n$ be as in the definition of $\mathrm{Kum}_s$, so that the image  of $\mathrm{Kum}_s(P\otimes\zeta_s)$ in  
$H^1(H_{c\ell}, J_s[p^n]^\ord \otimes\zeta_s)$ is given by the cocycle
$$
b_n(\sigma)=(\Theta^{-1}(\sigma)Q_n^\sigma-Q_n)\otimes\zeta_s.
$$
Let $Q\mapsto\widetilde{Q}$ denotes the reduction at $w$ of points  on $J_s(L)^\ord$, where $L$ is large enough that $Q_n\in J_s(L)^\ord$, and extend this to $J_s(L)^\ord\otimes\zeta_s$ by 
$$
\widetilde{Q\otimes\zeta_s}=\widetilde{Q}\otimes\zeta_s.
$$
By hypothesis $\widetilde{P}=0$, and so $p^n\widetilde{Q}_n=0$.  Enlarging $L$ if needed,  we may find $R_n\in J_s(L)^\ord$  such that 
$$
p^nR_n=0\hspace{1cm}
\widetilde{R}_n=\widetilde{Q}_n.
$$  
Replacing $Q_n$ by $Q_n-R_n$ we may therefore assume that $\widetilde{Q}_n=0$.  This implies that for all $\sigma$ in the decomposition group of $w$
$$
\widetilde{b_n(\sigma)} = (\Theta^{-1}(\sigma)\widetilde{Q}_n^\sigma-\widetilde{Q}_n)\otimes\zeta_s=0.
$$
By the injectivity of the reduction map on torsion, we find that the restriction of $b_n$ to the decomposition group of $w$ is trivial.  As this holds for all $n$, 
$$
\mathrm{loc}_w(\mathrm{Kum}_s(P\otimes\zeta_s))=0.
$$

We have now shown that $x_{c\ell,s}$ and $\Frob_v(x_{c,s})$ have the same image under
 (\ref{jacobian reduction}), hence also under (\ref{residual kummer}), and hence also under the composition (\ref{local kummer}).  Therefore
$$
\mathrm{loc}_w(\mathfrak{X}_{c\ell,s})=
\mathrm{loc}_w(\mathrm{Kum}_s(x_{c\ell,s}))=
\mathrm{loc}_w(\mathrm{Kum}_s( \Frob_v(x_{c,s})  ))=
\mathrm{loc}_w(\Frob_v(\mathfrak{X}_{c,s})).
$$
The claim is now immediate from the construction of $\mathfrak{X}_c$ from $\mathfrak{X}_{c,s}$.
\end{proof}

Let $W_N$ be the usual Atkin-Lehner automorphism of $X_s$, defined on elliptic curves with $\Phi_s$ level structure by
$$
W_N\cdot (E,C,P)= (E/C, \ker(f^\vee) ,f(P))
$$ 
where  $f:E\map{}E/C$ is the quotient map and $f^\vee$ is the dual isogeny.  Note that $W_N^2=\langle N\rangle$. The induced (Albanese) action on $J_s$ commutes with the nebentype operators, the  Hecke operators $T_\ell$ with $\ell\nmid Np$, and the operator  $U_p$.  In particular $W_N$ commutes with the ordinary projector $e^\ord$, and so there is an induced action (still denoted $W_N$) on $\mathbf{Ta}^\ord$ satisfying $W_N^2=[N]$.  Note that this action need not be $\hecke^\ord$-linear, as $W_N$ does not commute with the operators $U_\ell$ for $\ell\mid N$.

\begin{Lem}\label{sign}
There is an $R$-linear automorphism $W_N$ of $\FT$ making the diagram (of $\Lambda$-modules)
$$
\xymatrix{
 \mathbf{Ta}^\ord \ar[r]^{W_N}  \ar[d]  & \mathbf{Ta}^\ord\ar[d]   \\
\FT \ar[r]^{ W_N } & \FT
}
$$
commute.  Furthermore there is a choice of $w=\pm 1$ for which the action of $W_N$ on $\FT$ is given by $W_N=w\Theta(-N)$.
\end{Lem}

\begin{proof}
Let $\mathfrak{p}$ be an arithmetic prime of weight two and nontrivial wild character of conductor $p^s$.  By \cite[Theorem 3.1(i)]{hida} the $\Gamma^{p^s}$ coinvariants of $\mathbf{Ta}^\ord$ are precisely $\Ta_p^\ord(J_s)$, and so the natural map $\mathbf{Ta}^\ord\map{} V_\mathfrak{p}$ factors as
$$
\mathbf{Ta}^\ord\map{} \Ta_p(J_s)\otimes_{\hecke^\ord}F_\mathfrak{p}\map{}V_\mathfrak{p}.
$$
The modular form $g_\mathfrak{p}$ is new of level $\Phi_s$ and so the Galois representation $V_\mathfrak{p}$ appears with multiplicity one as a summand of $\Ta_p(J_s)\otimes_{\hecke^\ord}F_\mathfrak{p}$. As the automorphism $W_N$ commutes with the Galois action it must preserve this summand.  Furthermore, as $V_\mathfrak{p}$ is absolutely irreducible the action of $W_N$ on this summand is through a scalar in $F_\mathfrak{p}$.  We now have a commutative diagram
$$
\xymatrix{
 \mathbf{Ta}^\ord \ar[r]^{W_N}  \ar[d]  & \mathbf{Ta}^\ord\ar[d]   \\
V_\mathfrak{p} \ar[r]^{ W_N } & V_\mathfrak{p}
}
$$
of $\Lambda$-modules in which the bottom arrow is $\hecke^\ord$-linear.  

For any $h\in \hecke^\ord$ and $t\in \mathbf{Ta}^\ord$ we have shown that the image of $W_N ht-hW_Nt$ in $V_\mathfrak{p}$ is trivial.  As this holds for all arithmetic primes of weight two and nontrivial wild character, Lemma \ref{ring lemma} implies that $W_N ht-hW_Nt$ has trivial image in $\FT$.  From this it follows that the automorphism $W_N$ of $\mathbf{Ta}^\ord$ factors through  to an $\hecke^\ord$-linear automorphism of $\FT$.   By \cite[Theorem 2.1]{hida} $\FT\otimes_R\mathcal{K}$ is an irreducible Galois  representation, and by Proposition \ref{Gorenstein} complex conjugation acts on $\FT\otimes_R\mathcal{K}$ with distinct eigenvalues $\pm 1$.  Combining these, it follows that $\FT\otimes_R\mathcal{K}$ is absolutely irreducible.  Therefore $W_N$ must act on $\FT$ through a scalar $\lambda\in R$ satisfying $\lambda^2=[N]$.  But  $\Theta^2(-N)=[N]$ in $R$, and so for some choice of sign $w=\pm 1$ we have $W_N=w \cdot \Theta(-N)$  as operators on $\FT$.  
\end{proof}

\begin{Rem}\label{functional ambiguity}
Note that as $W_N:\FT\map{}\FT$ does not depend on the choice of $\Theta$ made in Remark \ref{twist remark}, neither does $w\Theta(-N)$.  The two possible choices of $\Theta$ differ by $\omega^{\frac{p-1}{2}}$, and making a different choice multiplies  $w$ by $\omega^{\frac{p-1}{2}}(-N)$.
\end{Rem}

\begin{Prop}\label{parity}
Suppose $\tau\in \Gal(H_c/\Q)$ acts nontrivially on $K$. There is a $\sigma\in \Gal(H_c/K)$ such that 
$$
\mathfrak{X}_c^\tau =w  \cdot \mathfrak{X}_c^\sigma
$$
where $w=\pm 1$ is defined by Lemma \ref{sign}.
\end{Prop}

\begin{proof}
We may assume that $\tau$ is the restriction of  the complex conjugation determined by the fixed embedding $\overline{\Q}\hookrightarrow\C$.  Set $\mathfrak{N}_{cp^s}=\mathfrak{N}\cap\co_{cp^s}$ so that the $N$-isogeny of complex tori $f:E_{c,s}(\C)\map{}(E_{c,s}/\mathfrak{n}_{c,s})(\C)$ is identified with
$$
\C/\co_{cp^s}\map{}\C/ \mathfrak{N}_{cp^s}^{-1},
$$
and the dual isogeny $f^\vee$ is identified with
$$
\C/ \mathfrak{N}_{cp^s}^{-1}\map{}\C/N^{-1}\co_{cp^s}.
$$
If we choose $x\in \widehat{K}^\times$ such that $x^{-1} \co_{cp^s}=\mathfrak{N}_{cp^s}$  and let $\sigma\in \Aut(\C/K)$ be such that the restriction of $\sigma$ to the maximal abelian extension of $K$ agrees with the Artin symbol of $x$, then the main theorem of complex multiplication \cite[Theorem 5.4]{shimura} gives the left square of the commutative diagram 
$$
\xymatrix{
 (E^\sigma_{c,s}/\mathfrak{n}_{c,s}^\sigma) (\C) \ar[r]\ar[d]_{(f^\vee)^\sigma} &  \C/x^{-1}\mathfrak{N}_{cp^s}^{-1}  \ar[r]^= \ar[d] &  \C/\tau(\co_{cp^s})  \ar[r]\ar[d]& E^\tau_{c,s} (\C) \ar[d]_{f^\tau} \\
E_{c,s}^\sigma (\C) \ar[r] &\C/x^{-1}N^{-1}\co_{cp^s} \ar[r]^= &  \C/\tau(\mathfrak{N}_{cp^s}^{-1}) \ar[r]
& (E_{c,s}/\mathfrak{n}_{c,s})^\tau(\C)
}
$$
in which all horizontal arrows are isomorphisms (the square on the right arises from an elementary argument using the Weierstrass $\wp$-function).  This implies that 
\begin{equation}\label{AL N moduli}
W_N\cdot (E_{c,s},\mathfrak{n}_{c,s})^\sigma \iso (E_{c,s}^\sigma/\mathfrak{n}_{c,s}^\sigma,\ker(f^\vee)^\sigma)\iso   (E^\tau_{c,s},\ker(f)^\tau )   \iso  (E_{c,s},\mathfrak{n}_{c,s})^\tau
\end{equation}
as elliptic curves with $\Gamma_0(N)$ level structure.    The elliptic curves $E_{c,s}^\sigma/\mathfrak{n}_{c,s}^\sigma$ and $E_{c,s}^\tau$ inherit from $E_{c,s}$  $\Gamma_1(p^s)$ level structures $f(\pi_{c,s})^\sigma$ and $\pi_{c,s}^\tau$ defined as the images of $\pi_{c,s}$ under the maps
$$
\ker(j_{c,s})\map{\sigma}E_{c,s}^\sigma \map{f^\sigma}E_{c,s}^\sigma/\mathfrak{n}_{c,s}^\sigma
\hspace{1cm}
\ker(j_{c,s})\map{\tau}E_{c,s}^\tau.
$$
Using the isomorphisms in the top row of the diagram above, these images correspond to the images of $c\varpi$ under
$$
\co_c/\co_{cp^s} \map{x^{-1}}   K / \co_{cp^s} \hspace{1cm}
\co_c/\co_{cp^s} \map{\tau}K/  \tau(\co_{cp^s}),
$$
and if we demand that $x$ has component $x_p=-1$ in $K\otimes\Q_p$ then both images are $-c\varpi$.  Thus
$$
 (E_{c,s}/\mathfrak{n}_{c,s}, f(\pi_{c,s}))^\sigma \iso   (E_{c,s},\pi_{c,s})^\tau   
$$
as elliptic curves with $\Gamma_1(p^s)$ level structure.  Combining this with (\ref{AL N moduli}) we find
$$
W_N\cdot h_{c,s}^\sigma=  h_{c,s}^\tau.
$$ 

Using the fact that $W_N$ commutes with $e^\ord$ and with $e_{k+j-2}$, it follows that $W_N\cdot y_{c,s}^\sigma=  y_{c,s}^\tau$.  Fix a set $S\subset\Gal(L_{c,s}/H_c)$ of representatives for $\Gal(H_{cp^s}/H_c)$.  Let $W_N$ act on $J_s(L_{c,s})^\ord\otimes\zeta_s$ by $W_N(P\otimes\zeta_s)=(W_NP) \otimes\zeta_s$.  We  compute
\begin{eqnarray*}
x^\tau_{c,s} &=& \sum_{\eta\in  S}  \Theta^{-1}( \eta) (y^\eta_{c,s})^\tau \otimes\zeta_s^\tau \\
&=&  
  \sum_{\eta\in  S} \Theta^{-1}(\eta) (y^\tau_{c,s})^{\tau\eta\tau^{-1}} \otimes\zeta_s^\tau \\
 &=& 
   W_N \sum_{\eta\in  \tau S\tau^{-1}}  \Theta^{-1}( \eta) (y^\sigma_{c,s})^{\eta} \otimes\zeta_s^\tau \\
 &=& 
   \Theta(\tau^{-1}) W_N \sum_{\eta\in  \tau S\tau^{-1}}  \Theta^{-1}( \eta) (y^\eta_{c,s})^{\sigma} \otimes\zeta_s \\
    &=& 
   \Theta(\sigma\tau^{-1}) W_N \big( \sum_{\eta\in  \tau S\tau^{-1}}  \Theta^{-1}( \eta) y^\eta_{c,s} \otimes\zeta_s \big)^\sigma \\
   &=&
   \Theta(\sigma\tau^{-1}) W_N x_{c,s}^\sigma.
   \end{eqnarray*}
 As  $\sigma$ acts on $\mu_{p^\infty}$ through the Artin symbol of $N_{K/\Q}(x)$, an idele of norm $N^{-1}$, we have $\epsilon_\cyc(\sigma)=N^{-1}$.  Identifying $\Theta$ with a character on $\Z_p^\times$ by factoring through the cyclotomic character, this says
 $\Theta(\sigma)=\Theta(N)^{-1}$, and similarly $\Theta(\tau)=\Theta(-1)$.  Therefore
$$
x_{c,s}^{\tau}=   \Theta^{-1}(-N)W_N x_{c,s}^\sigma.
$$
By the $G_\Q$ and $\Aut(J_s)$ equivariance of the twisted Kummer map this equality holds with
$x_{c,s}$ replaced by $\mathfrak{X}_{c,s}$.  Passing to the limit in $s$ and using Lemma \ref{sign}
(and viewing $W_N$ as an operator on $\FT^\dagger$ using the identification $\FT^\dagger\iso\FT$ of underlying $R$-modules)  it  follows that
$$
\mathfrak{X}_c^\tau = \Theta^{-1}(-N) W_N\cdot \mathfrak{X}_c^\sigma.
$$
As $W_N$ acts as $w\Theta(-N)$ on $\FT^\dagger$ by Lemma \ref{sign}, we are done.
\end{proof}

\begin{Prop}\label{functional equation}
Let $\mathfrak{p}$ be an arithmetic prime of $R$  and let $\co_\mathfrak{p}$ denote the ring of integers of $F_\mathfrak{p}$.  The Mazur-Tate-Teitelbaum \cite{MTT}  $p$-adic $L$-function 
$$L_p(g_\mathfrak{p},\cdot)\in\co_\mathfrak{p}[[\Z_p^\times]],$$
viewed as a function on characters $\Z_p^\times\map{}\overline{\Q}_p^\times$, satisfies the functional equation
\begin{equation}\label{functional display}
L_p(g_\mathfrak{p}, \chi ) 
= - w \chi^{-1}(-N)  \Theta_\mathfrak{p}(-N) \cdot  L_p(g_\mathfrak{p},  \chi^{-1} [\cdot]_\mathfrak{p} )
\end{equation}
where $w=\pm 1$ is defined by Lemma \ref{sign}. In particular, taking $\chi=\Theta_\mathfrak{p}$  and using  $\Theta_\mathfrak{p}^2  =  [\cdot]_\mathfrak{p}$ gives
$$
L_p(g_\mathfrak{p}, \Theta_\mathfrak{p} ) 
=  -w \cdot  L_p( g_\mathfrak{p},  \Theta_\mathfrak{p} ).
$$
\end{Prop}

\begin{proof}
According to \cite[\S 3.4]{EPW} there is a two-variable $p$-adic $L$-function $L_p\in R[[\Z_p^\times]]$ whose image under
$$
R[[\Z_p^\times]] \map{} \co_\mathfrak{p}[[\Z_p^\times]]
$$
for any arithmetic prime $\mathfrak{p}\subset R$ agrees (up to a nonzero $p$-adic period) with the one variable $p$-adic $L$-function $L_p(g_\mathfrak{p},\cdot)$ of \cite{MTT}.  Such a two variable $p$-adic $L$-function also appears in the work of Greenberg-Stevens \cite{GS}.  If we take $\mathfrak{p}$ to be an arithmetic prime of weight $2$ (so that $g_\mathfrak{p}$ has character $[\cdot]_\mathfrak{p}$)  then \cite[ \S I.17, Corollary 2]{MTT} gives
\begin{equation}\label{MTT special}
L_p(g_\mathfrak{p},\chi) = -  \chi^{-1}(-N) [N]_\mathfrak{p} \cdot
L_p(g_\mathfrak{p}^*, \chi^{-1}  [\cdot]_\mathfrak{p})
\end{equation}
where 
$g_\mathfrak{p}^* = \langle N\rangle^{-1} W_N \cdot g_\mathfrak{p}.$ 
Here $W_N$ acts on modular forms by the contravariant action on
$$
S_2(\Phi_s, F_\mathfrak{p})\iso 
H^0(X_{s/F_\mathfrak{p}} ,\Omega^1_{X_s/F_\mathfrak{p}} )
$$
(our $\langle N\rangle^{-1} W_N$ is the $w_Q$ of [\emph{loc. cit.}]).
%as follows (as always $s$ is the larger of $1$ and the order at $p$ of the conductor of the wild character $\psi_\mathfrak{p}$).  Fix $x,y,z,w\in\Z$ such that $Nxw-p^syz=1$, and set
%$$
%W'_N=\left( \begin{matrix} Nx & y \\ Np^sz & N w \end{matrix} \right)
%$$
%so that $W'_N=\langle w\rangle W_N=\langle Nx \rangle W_N
But then 
$$
g_\mathfrak{p}^* = [N]_\mathfrak{p}^{-1} \cdot w \Theta_\mathfrak{p}(-N) \cdot g_\mathfrak{p}  \\
$$
as the eigenvalue of $W_N$ acting on $g_\mathfrak{p}$ agrees with the eigenvalue $w\Theta_\mathfrak{p}(-N)$ of $W_N$ acting on the corresponding factor $V_\mathfrak{p}$ of the $p$-adic Tate module of $J_s$.  Combining this with (\ref{MTT special}) we find that there are infinitely many  arithmetic primes $\mathfrak{p}\subset R$ such that (\ref{functional display}) holds for every character $\chi$.

Now fix a character $\chi:\Z_p^\times\map{}\co_F^\times$, possibly of infinite order, and define $R$-module maps 
\begin{eqnarray*}
\pi_\chi:R[[\Z_p^\times]]\map{}R & & z\mapsto \chi(z) \\
\pi_\chi^*:R[[\Z_p^\times]]\map{}R & & z\mapsto   \chi^{-1}(z)\cdot [z]
\end{eqnarray*}
for $z\in\Z_p^\times$.
Composing these with $R\map{}\co_\mathfrak{p}$ for any arithmetic prime $\mathfrak{p}$ gives two more maps 
$$
\pi_{\chi,\mathfrak{p}}, \pi_{\chi,\mathfrak{p}}^* :R[[\Z_p^\times]]\map{}\co_\mathfrak{p},
$$
and by what we have proved the equality 
$$
\pi_{\chi,\mathfrak{p}} (L_p)=-w\chi^{-1}(-N)\Theta_\mathfrak{p} (-N) \pi_{\chi,\mathfrak{p}}^*(L_p)
$$
holds for all arithmetic primes of weight two.  By Lemma \ref{ring lemma} we must have
$$
\pi_{\chi} (L_p)=-w\chi^{-1}(-N)\Theta (-N) \pi_{\chi}^*(L_p).
$$
This means that for \emph{every} arithmetic prime $\mathfrak{p}$ the equality (\ref{functional display}) holds for all characters $\chi:\Z_p^\times\map{}\co_F^\times$, and as $L_p(g_\mathfrak{p},\cdot)$ is determined by its values on such characters we conclude that (\ref{functional display}) holds for all arithmetic primes and all characters.
\end{proof}

%%%%%%%%%%%%%%%%%%%%%%%%%%%%%%%%%%%%%%%%%%%%%%%%%%%%%%%%%%%%%%%%%%%

\subsection{Selmer groups}
\label{ss:selmer}

%%%%%%%%%%%%%%%%%%%%%%%%%%%%%%%%%%%%%%%%%%%%%%%%%%%%%%%%%%%%%%%%%%%

\begin{Prop}\label{Prop:ord}
Let $v$ be a place of $\overline{\Q}$ above $p$ and let $I_v\subset D_v\subset G_\Q$ be the inertia and decomposition groups of $v$.  Let $\eta_v:D_v/I_v\map{}R^\times$ be the character taking the  arithmetic Frobenius to $U_p$. There is a short exact sequence of $R[D_v]$-modules 
\begin{equation}\label{local filtration}
0\map{}F^+_v(\FT)\map{}\FT\map{}F^-_v(\FT)\map{}0
\end{equation}
such that $F^+_v(\FT)$ and $F^-_v(\FT)$ are free of rank one over $R$, $D_v$ acts on $F^-_v(\FT)$ through $\eta_v$, and $D_v$ acts on $F^+_v(\FT)$ through  $\eta_v^{-1}\epsilon_\cyc [ \epsilon_\cyc]$.
\end{Prop}

\begin{proof}
When $k+j\not\equiv 2\pmod{p-1}$ the short exact sequence is that of  \cite[Proposition 1.5.2(iii)]{nek} together with the final statement of Proposition \ref{Gorenstein} for the isomorphism $R\iso F^-_v(\FT)$. When $k+j-2\equiv 0\pmod{p-1}$ the short exact sequence is that of  \cite[Proposition 1.5.4]{nek}.  In either case the description of the action of $D_v$ follows from \cite[Theorem 2.6(d)]{GS}.
\end{proof}

Twisting by $\Theta^{-1}$ and tensoring the exact sequence (\ref{local filtration}) with $R_\mathfrak{p}$ or $F_\mathfrak{p}$ (for any arithmetic prime $\mathfrak{p}\subset R$), yields an exact sequence of $D_v$-modules
$$
0\map{}F^+_v(M)\map{}M\map{}F^-_v(M)\map{}0
$$
for $M$ any one of $\FT_\mathfrak{p}$, $V_\mathfrak{p}$, $\FT^\dagger_\mathfrak{p}$, or
$V^\dagger_\mathfrak{p}$.

\begin{Def}\label{Greenberg def}
Let $L$ be a finite extension of $\Q$.  For each prime $v$ of $L$ let $L_v^\unr$ be the maximal unramified extension of $L_v$.  Let $M$ be any $G_\Q$-module for which we have defined  $F_v^-(M)$. For any place $v$ of $L$ define the \emph{strict Greenberg local condition}  (compare with \cite[\S 2]{greenberg} or \cite[\S 9.6.1]{selmer_complexes}
$$
H^1_\Gr(L_v,M)=
\left\{\begin{array}{ll}
\mathrm{ker}\big( H^1(L_v,M)\map{}H^1(L_v^\unr,M) \big)
&\mathrm{if\ }v\nmid p \\  \\
\mathrm{ker}\big( H^1(L_v,M)\map{}H^1(L_v,F_v^-(M)) \big)
&\mathrm{if\ }v\mid p
\end{array}\right.
$$
and the \emph{strict Greenberg Selmer group} 
$$
\Sel_\Gr(L,M) =\mathrm{ker}\big( H^1(L,M)\map{}  \prod_v H^1(L_v,M)/H^1_\Gr(L_v,M)\big)
$$
where the product is over all finite places of $L$.
\end{Def}

\begin{Def}\label{Def:exceptional}
Suppose $\mathfrak{p}\subset R$ is an arithmetic prime of weight $r$ and let $\alpha_\mathfrak{p}$ be the image of $U_p$ under $R\map{}F_\mathfrak{p}$. We will say that $\mathfrak{p}$ is  \emph{exceptional} if $r=2$, $\psi_\mathfrak{p}$ is the trivial character, and $\alpha_\mathfrak{p} =\pm 1.$
\end{Def}

\begin{Lem}\label{H^0 terms}
Let $L/\Q$ be a finite extension and fix a place $v$ of $\overline{\Q}$ above $p$.  If $\mathfrak{p}$ is an arithmetic prime of $R$ which is not exceptional then
$$
H^0(L_v,F_v^-(V_\mathfrak{p}^\dagger))=0.
$$
Furthermore, if $\tilde{L}_v$ is any finite extension of a ramified $\Z_p$-extension of $L_v$ then
$$
H^0(\tilde{L}_v,F_v^-(\FT^\dagger))=0.
$$
\end{Lem}

\begin{proof}
Fix an arithmetic prime $\mathfrak{p}$ of $R$ of weight $r$.  Proposition \ref{Prop:ord} implies that the action of $\Gal(\overline{\Q}_p/\Q_p)$ on $F_p^-(V_\mathfrak{p}^\dagger)$ is through the $F_\mathfrak{p}^\times$ valued character $\eta_{v,\mathfrak{p}}\Theta_\mathfrak{p}^{-1}$
 where $\eta_{v,\mathfrak{p}}$ takes the arithmetic Frobenius to $\alpha_\mathfrak{p}$.  The claim is that 
$$
 \eta_{v,\mathfrak{p}}\Theta_\mathfrak{p}^{-1}\mathrm{\ of\ finite\ order\ }\implies \mathfrak{p} \mathrm{\ exceptional}.
 $$
Indeed,  if $ \eta_{v,\mathfrak{p}}\Theta_\mathfrak{p}^{-1}$ is of finite order then $\eta_{v,\mathfrak{p}}$ becomes trivial when restricted to the Galois group of a finite extension of $L_v(\mu_{p^\infty})$.  As $\eta_{v,\mathfrak{p}}$ is unramified, it follows that $\eta_{v,\mathfrak{p}}$ is of finite order and so $\alpha_{\mathfrak{p}}$ is a root of unity.  But then also $\Theta_\mathfrak{p}$ is of  finite order, and  as 
 $$
\Theta_\mathfrak{p}(\sigma) = \epsilon_\tame^{\frac{k+j}{2} -1 }(\sigma) \cdot \epsilon_\wild^{\frac{r}{2}-1}(\sigma) \cdot \psi_\mathfrak{p}(\epsilon_\wild^{1/2}(\sigma ))
 $$
we must have $r=2$. By \cite[\S I.12 Case II]{MTT} the eigenform $g_\mathfrak{p}$ of weight $2$ attached to $\mathfrak{p}$ is a newform of level $Np$ and trivial character with $\alpha_\mathfrak{p}=\pm 1$.  Thus $\mathfrak{p}$ is exceptional.

The first claim of the lemma is now immediate, as $H^0(L_v,F_v^-(V_\mathfrak{p}^\dagger))\not=0$ implies that $ \eta_{v,\mathfrak{p}}\Theta_\mathfrak{p}^{-1}$ is of finite order.  For the second claim, fix a weight $2$ arithmetic prime $\mathfrak{p}$  which is not exceptional, so that  $\eta_{v,\mathfrak{p}}\Theta_\mathfrak{p}^{-1}$ has infinite order.  As $\Theta_\mathfrak{p}$ is of finite order,  $\eta_{v,\mathfrak{p}}$ is  unramified and of infinite order and so cannot become trivial over any extension of $L_v$ whose maximal unramified subfield is finite over $\Q_p$.  In particular $\eta_{v,\mathfrak{p}}$ has nontrivial restriction to any finite extension of $\tilde{L}_v$, and so $ \eta_{v,\mathfrak{p}}\Theta_\mathfrak{p}^{-1}$ has nontrivial restriction to $\tilde{L}_v$.  Therefore $\eta_v\Theta^{-1}$ has nontrivial restriction to $\tilde{L}_v$, and so $H^0(\tilde{L}_v,F_v^-(\FT^\dagger))=0$.
\end{proof}

If $M$ is either  $\FT^\dagger$ or  $V^\dagger_\mathfrak{p}$ and $L/\Q$ is a finite extension, one also has a family of \emph{extended Selmer groups} $\tilde{H}^i_f(L,M)$  defined by Nekov{\'a}{\v{r}} \cite{selmer_complexes} using similar local conditions to those defining $\Sel_\Gr$, but with the local conditions imposed on the level of cochain  complexes rather than on cohomology.  For $i=1$ Nekov{\'a}{\v{r}}'s extended Selmer group sits in an exact sequence \cite[Lemma 9.6.3]{selmer_complexes}
$$
0\map{}\bigoplus_{v\mid p}H^0(L_v,F_v^-(M))  \map{} \tilde{H}^1_f(L, M)\map{}\Sel_\Gr(L,M)\map{}0.
$$
In particular Lemma \ref{H^0 terms}  implies
\begin{equation}\label{big nek green}
\tilde{H}^1_f(L,\FT^\dagger)\iso\Sel_\Gr(L,\FT^\dagger)
\end{equation}
and, if $\mathfrak{p}$ is an arithmetic prime of $R$ which is not exceptional, 
\begin{equation}\label{nek-green compare}
\tilde{H}^1_f(L,V^\dagger_\mathfrak{p}) \iso \Sel_\Gr(L,V^\dagger_\mathfrak{p}).
\end{equation}
If $\mathfrak{p}$ has even weight then according to \cite[Proposition 12.5.9.2]{selmer_complexes} one also has an exact sequence
$$
0\map{}\bigoplus_{v\mid p}H^0(L_v,F_v^-(V_\mathfrak{p}^\dagger))  \map{} \tilde{H}^1_f(L, V_\mathfrak{p}^\dagger)\map{}H_f^1(L,V_\mathfrak{p}^\dagger)\map{}0
$$
in which $H_f^1(L,V_\mathfrak{p}^\dagger)$ is the Bloch-Kato Selmer group.  In particular
\begin{equation}\label{bloch-kato compare}
\Sel_\Gr(L,V_\mathfrak{p}^\dagger)=H_f^1(L,V_\mathfrak{p}^\dagger)
\end{equation}
as both are equal to the image of $\tilde{H}^1_f(L,V_\mathfrak{p}^\dagger)\map{}H^1(L,V_\mathfrak{p}^\dagger)$.

\begin{Prop}\label{Prop:greenberg}
For any positive integer $c$ prime to $N$  there is a nonzero $\lambda\in R$ (which may depend on $c$) such that
$$
\lambda\cdot \mathfrak{X}_c\in \Sel_\Gr(H_c,\FT^\dagger).
$$
If  $N$ is prime to $\mathrm{disc}(K)$ then we may take $\lambda=1$.
\end{Prop}

\begin{proof}
For any place $v$ of $H_c$ and any $G_\Q$-module $M$ denote by
$$
\mathrm{loc}_v:H^1(H_c,M)\map{}H^1(H_{c,v},M)
$$
the localization map.  If $v$ is a finite prime of $H_c$ not dividing $Np$, then the  unramifiedness of $\mathfrak{X}_c$ at $v$ is part of Definition \ref{Def:big}.  

Now suppose $v\mid Np$ and choose a place $w$ of $\overline{\Q}$ above $v$.  Let $\mathfrak{p}$ be an arithmetic prime of weight $2$, and let $s=\max\{1,\ord_p(\mathrm{cond}(\psi_\mathfrak{p})) \}$ so that  the natural map $\mathbf{Ta}^\ord\map{}V_\mathfrak{p}$ factors through  $\Ta_p^\ord(J_s)$.  Let $\mathfrak{X}_{c,\mathfrak{p}}$ denote the image of $\mathfrak{X}_c$ in $H^1(H_c,V^\dagger_\mathfrak{p})$.  If we set $L_{c,s}=H_{cp^s}(\mu_{p^s})$ then $V_\mathfrak{p}^\dagger\iso V_\mathfrak{p}$ after restriction to $L_{c,s}$, and directly from the construction we see that the restriction of  $\mathfrak{X}_{c,\mathfrak{p}}$ to $H^1(L_{c,s}, V_\mathfrak{p}^\dagger)$ lies in the image of the composition
$$
J_s(L_{c,s})^\ord  \map{}H^1(L_{c,s},\Ta_p^\ord(J_s)) \map{}H^1(L_{c,s},V_\mathfrak{p})\iso H^1(L_{c,s},V_\mathfrak{p}^\dagger)
$$
where the first arrow is the usual (untwisted) Kummer map, and is  therefore contained in in the Bloch-Kato Selmer group by  \cite[Example 3.11]{BK} (see also  \cite[Proposition 1.6.8]{rubin}).  By (\ref{bloch-kato compare}) the restriction of $\mathfrak{X}_{c,\mathfrak{p}}$ to $L_{c,s}$ lies in $\Sel_\Gr(L_{c,s}, V_\mathfrak{p}^\dagger)$.  If  $v\mid N$ then the fact that $L_{c,s,w}/H_{c,v}$ is unramified implies that 
\begin{equation}\label{localization in greenberg}
\mathrm{loc}_v ( \mathfrak{X}_{c,\mathfrak{p}} ) \in H^1_\Gr(H_{c,v},V_\mathfrak{p}^\dagger).
\end{equation}
If  $v\mid p$ then the restriction map 
$$
H^1(H_{c,v},F_v^-(V_\mathfrak{p}^\dagger))  \map{}   H^1(L_{c,s,w},F_v^-(V_\mathfrak{p}^\dagger))
$$
is injective,  as the kernel
$$
H^1(L_{c,s,w}/H_{c,v}, H^0(L_{c,s,w},V_\mathfrak{p}^\dagger))
$$
is both an $F_\mathfrak{p}$-vector space and is annihilated by $[L_{c,s,w}:H_{c,v}]$.
Therefore (\ref{localization in greenberg}) holds in this case as well.
We have now shown that 
\begin{equation}\label{often in greenberg}
\mathfrak{X}_{c,\mathfrak{p}} \in \Sel_\Gr(H_{c,v},V^\dagger_\mathfrak{p})
\end{equation}
for all arithmetic primes of weight $2$.

Suppose $v\mid p$.  For any arithmetic prime $\mathfrak{p}$ we may, by Lemma \ref{DVR},  fix a generator $\pi$ of the maximal ideal of  $R_\mathfrak{p}$.  The exactness of
$$
0\map{}F^-_v(\FT^\dagger_\mathfrak{p})   \map{\pi}   F^-_v(\FT^\dagger_\mathfrak{p})   \map{}   F^-_v(V_\mathfrak{p}^\dagger)\map{}0
$$
shows that the natural map
\begin{equation}\label{green spec}
\frac{H^1(H_{c,v},F^-_v(\FT^\dagger))_\mathfrak{p}}   {\mathfrak{p}\cdot H^1(H_{c,v},F^-_v(\FT^\dagger))_\mathfrak{p}}  \map{}H^1(H_{c,v},F^-_v(V_\mathfrak{p}^\dagger))
\end{equation}
is injective.  By (\ref{often in greenberg}) the image  of $\mathfrak{X}_{c}$ in the right hand side of  (\ref{green spec}) is trivial for infinitely many $\mathfrak{p}$.  Combining \cite[Proposition 4.2.3]{selmer_complexes} and \cite[Theorem 7.1.8(iii)]{neuk} the $R$-module $H^1(H_{c,v},F_v^-(\FT^\dagger))$ is finitely generated.  Therefore Lemma \ref{ring lemma} applies and the image of $\mathfrak{X}_c$ under
\begin{equation}\label{greenberg loc}
H^1(H_c,\FT^\dagger)  \map{}  H^1(H_{c,v},F_v^-(\FT^\dagger))
\end{equation}
is a torsion element. Let $S_n\subset H^1(H_{cp^n,w},F_v^-(\FT^\dagger))$ be the  $R$-torsion submodule. By the Euler system relations of Proposition \ref{Prop:ESrelations} and the  discussion above (which holds equally well with $c$ replaced by $cp^n$),   the image of $\mathfrak{X}_{c}$  under (\ref{greenberg loc}) lies in the image of $\mil S_n\map{}S_0$, where the inverse  limit is with respect to corestriction.  Set 
$$
\mathbf{V}=F_v^-(\FT^\dagger)\otimes_{R}\mathcal{K}
$$ 
(recall that $\mathcal{K}$ is the fraction field of $R$)  and define $\mathbf{A}$ by the exactness of 
$$
0\map{}F^-_v(\FT^\dagger)\map{}\mathbf{V} \map{}\mathbf{A}\map{}0.
$$
The restriction of $\eta_v\Theta^{-1}$, the character giving the  Galois action on $\mathbf{V}$, to  $H_{cp^\infty,w}$ is nontrivial by Lemma \ref{H^0 terms}, and it follows that 
$$
S_n\iso H^0(H_{cp^n,w},\mathbf{A}).
$$  
If we pick  a $\sigma\in \Gal(\overline{\Q}_p/H_{cp^{\infty},w})$ such that $(\eta_v\Theta^{-1})(\sigma)\not=1$ then, as $\mathbf{A}\iso \mathcal{K}/ R$ as  $R$-modules,
$$
H^0(H_{cp^\infty,w},\mathbf{A})\subset \mathbf{A}[\sigma-1]\iso R/\big((\eta_v\Theta^{-1})(\sigma)-1\big).
$$
 In particular  $H^0(H_{cp^\infty,w},\mathbf{A})$ is finitely generated as an $R$-module, and so for a sufficiently large $m$ the restriction map $S_m\map{}S_n$ is an isomorphism for all $n\ge m$.  This implies that the image of corestriction $S_n\map{}S_m$ is divisible by $p^{n-m}$.  As $S_m$, being finitely generated over $R$, has no nontrivial $p$-divisible submodules, $$\mil S_n=0.$$  This proves that  the image of $\mathfrak{X}_{c}$  under (\ref{greenberg loc}) is trivial, and so 
$$\mathrm{loc}_v(\mathfrak{X}_c) \in H^1_\Gr(H_{c,v},\FT^\dagger).$$

Now suppose $v\mid N$. We know from above that the image $\mathfrak{X}_c$ under
$$
\frac{H^1(H_{c,v}^\unr,\FT^\dagger)_\mathfrak{p}}{\mathfrak{p}H^1(H_{c,v}^\unr,\FT^\dagger)_\mathfrak{p}}  \map{}H^1(H_{c,v}^\unr, V_\mathfrak{p}^\dagger)
$$
is trivial for infinitely many arithmetic primes $\mathfrak{p}$.  The finite generation of the $R$-module $H^1(H_{c,v}^\unr,\FT^\dagger)  $ is a consequence of \cite[Proposition 4.2.3]{selmer_complexes}  as above, provided one knows that $H^1(H_{c,v}^\unr, M)$ is finite for every finite Galois module $M$ of $p$-power order.  This follows from the proof of \cite[Theorem 7.1.8(iii)]{neuk}, the essential point being that one knows the finiteness of $H^i(B, \mu_{p^n})$ for every finite extension $B/H_{c,v}^\unr$ by passing to the limit over finite subfields in  \cite[Theorem 7.1.8(ii)]{neuk}.
As in the case $v\mid p$, this implies that the restriction of $\mathfrak{X}_c$ to $H^1(H_{c,v}^\unr,\FT^\dagger)$ is $R$-torsion.  Thus there is a nonzero $\lambda_v\in R$ such that 
$$
\mathrm{loc}_v(\lambda_v\mathfrak{X}_c)\in H^1_\Gr(H_{c,v},\FT^\dagger).
$$  
Taking $\lambda=\prod_{v\mid N} \lambda_v$ we then have $\lambda\mathfrak{X}_c\in\Sel_\Gr(H_c,\FT^\dagger)$.  If we assume that $N$ is prime to $\mathrm{disc}(K)$ then $v$ splits in $K$ and it follows that  $v$ is finitely decomposed in $H_{cp^\infty}$.   By the Euler system relations (Proposition \ref{Prop:ESrelations}) $\mathfrak{X}_c$ is a universal norm from the $\Z_p$-extension $H_{cp^\infty}/H_c$, hence  the image of $\mathfrak{X}_c$ in $H^1(H_c,\FT^\dagger/\mathfrak{m}^k\FT^\dagger)$ for every $k\ge 0$ ($\mathfrak{m}$ is the maximal ideal of $R$) is unramified at $v$ by \cite[Corollary B.3.5]{rubin}.  It follows that $\mathfrak{X}_c$ has trivial image in
$$
H^1(H^\unr_{c,v},\FT^\dagger)=\mil_k H^1(H^\unr_{c,v},\FT^\dagger/\mathfrak{m}^k\FT^\dagger),
$$
using  \cite[Lemma 4.2.2]{selmer_complexes} to justify passing to the limit, and so
$$
\mathrm{loc}_v(\mathfrak{X}_c) \in H^1_\Gr(H_{c,v},\FT^\dagger).
$$
\end{proof}

\begin{Rem}
The proof of Proposition \ref{Prop:greenberg} shows that $\mathfrak{X}_c$ satisfies the correct local conditions to lie in $\Sel_\Gr(H_c,\FT^\dagger)$, except possibly at places dividing both $N$ and $\mathrm{disc}(K)$.
\end{Rem}

%%%%%%%%%%%%%%%%%%%%%%%%%%%%%%%%%%%%%%%%%%%%%%%%%%%%%%%%%%%%%%%%%%%%%%

\section{Iwasawa theory}
\label{Iwasawa}

%%%%%%%%%%%%%%%%%%%%%%%%%%%%%%%%%%%%%%%%%%%%%%%%%%%%%%%%%%%%%%%%%%%%%

Throughout all of \S \ref{Iwasawa} we assume that $N$ and $\mathrm{disc}(K)$ are relatively prime and that $p\nmid\phi(N)$, where $\phi$ is Euler's function.   In particular for every positive integer $c$ prime to $N$ the big Heegner point $\mathfrak{X}_c$ of Definition \ref{Def:big} lies in $\Sel_\Gr(H_c,\FT^\dagger)$ by Proposition \ref{Prop:greenberg}

%%%%%%%%%%%%%%%%%%%%%%%%%%%%%%%%%%%%%%%%%%%%%%%%%%%%%%%%%%%%%%%%%%%%%%

\subsection{The vertical nonvanishing theorem}
\label{ss:vnt}

%%%%%%%%%%%%%%%%%%%%%%%%%%%%%%%%%%%%%%%%%%%%%%%%%%%%%%%%%%%%%%%%%%%%%%%%

Let $D_\infty$ be the anticyclotomic $\Z_p$-extension of $K$. Define $\mathcal{G}=\Gal(H_{p^\infty}/K)$
and let
$$
\mathcal{G}_\tors= \Gal(H_{p^\infty}/D_\infty)
$$
be the torsion subgroup of $\mathcal{G}$.  Keep $R$, $\FT$, and $\FT^\dagger$ as in  \S \ref{ss:construction}.  Let  $\mathfrak{m}$ be the maximal ideal of $R$.  The modular form $g$ of the introduction furnishes us with a surjection $R\map{}\co_F$ as in \S \ref{Hida} inducing an isomorphism of residue fields
$R/\mathfrak{m}\iso \co_F/\mathfrak{m}\co_F$.  Fix a character $\chi:\mathcal{G}_\tors\map{}\co_F^\times$ and set
$$
e_\chi=\sum_{g\in\mathcal{G}_\tors}\chi(g)\cdot g\in \co_F[\mathcal{G}_\tors].
$$ 
The proof of the following theorem, which follows closely the methods of Cornut and Vatsal \cite{cornut02,C-V,vatsal02},  will be given in \S \ref{ss:prnv}.

\begin{Thm}\label{RNV}
Let 
$$
\mathrm{Heeg}_s\subset  H^1(H_{p^s}, \FT^\dagger/\mathfrak{m}\FT^\dagger)
$$
be the $R[\Gal(H_{p^s}/K)]$-submodule generated by  the image of $\mathfrak{X}_{p^s}$.  As $s\to\infty$ the $R/\mathfrak{m}$ dimension of $e_\chi \mathrm{Heeg}_s$  grows without bound.  
\end{Thm}

\begin{Cor}\label{VNV}
Let $\mathfrak{p}\subset R$ be any arithmetic prime. For all $s\gg 0$,  the image of $e_\chi\mathfrak{X}_{p^s}$ in $\Sel_\Gr(H_{p^s}, V_\mathfrak{p}^\dagger)$ is nontrivial.
\end{Cor}

\begin{proof}[Proof of Corollary \ref{VNV}]
Let $T_\mathfrak{p}^\dagger$ denote the image of   $\FT^\dagger\map{}V_\mathfrak{p}^\dagger$.  As in \cite[Lemma 5.1.5]{howard-cycles} the torsion submodule of $H^1(H_{p^s},T_\mathfrak{p}^\dagger)$ has bounded order as $s\to\infty$. Thus if $e_\chi\mathfrak{X}_{p^s}$ had torsion image in $H^1(H_{p^s},T_\mathfrak{p}^\dagger)$ for all $s$ then $e_\chi \mathrm{Heeg}_s$ would have bounded dimension as $s\to\infty$ contradicting Theorem \ref{RNV}.  Thus $e_\chi\mathfrak{X}_{p^s}$ is nontrivial in $H^1(H_{p^s},V_\mathfrak{p}^\dagger)$ for some $s$, and then by the Euler system relations  of Proposition \ref{Prop:ESrelations} it must be nontrivial for all $s\gg 0$.
\end{proof}

%%%%%%%%%%%%%%%%%%%%%%%%%%%%%%%%%%%%%%%%%%%%%%%%%%%%%%%%%%%%%

\subsection{Proof of Theorem \ref{RNV}}
\label{ss:prnv}

%%%%%%%%%%%%%%%%%%%%%%%%%%%%%%%%%%%%%%%%%%%%%%%%%%%%%%%%%%%%%

For all $s>0$ define $$L_s^*=H_{p^{s+1}}(\mu_p)= L_{p^s,1}.$$  Set
$L_\infty^*=\cup L_s^*$ and $\mathcal{G}^*=\Gal(L_\infty^*/K)$.
Let 
$$
\mathcal{G}_\ram= \langle \Frob_{Q_1},\ldots, \Frob_{Q_t} \rangle \subset\mathcal{G}_\tors
$$ 
be the subgroup  generated by the Frobenius automorphisms of all ramified primes $Q_1,\dots,Q_t$ of $K$ which are prime to $p$ (all of which have order two and are linearly independent over $\Z/2\Z$).  Let $\mathfrak{D}=Q_1\cdots Q_t$ and $D=\mathrm{Norm}(\mathfrak{D})$, and define
$$
\mathcal{R}_\ram=\{\Frob_I\mid \mathfrak{D} \subset I\subset \co\} \subset\mathcal{G}^*
$$
where $\Frob_I\in\mathcal{G}^*$ is the Frobenius of $I$.  The quotient map $\mathcal{G}^* \map{}\mathcal{G}$ takes $\mathcal{R}_\ram$ bijectively to $\mathcal{G}_\ram$.
 Fix a subset $\mathcal{R}\subset \mathcal{G}^*$ whose elements represent the distinct cosets $\mathcal{G}_\tors/\mathcal{G}_\ram$, so that $\mathcal{G}^* \map{}\mathcal{G}$ takes the set
  $$
  \mathcal{R}_\tors\define \{\sigma\tau\mid \sigma\in \mathcal{R},\tau\in\mathcal{R}_\ram\} \subset\mathcal{G}^*
  $$
bijectively to $\mathcal{G}_\tors$.  Abbreviate 
$$
\Phi=\Phi_1=\Gamma_0(N)\cap\Gamma_1(p)
$$ 
and  write $X(\Phi)$ instead of $X_1$.  Set
$$
h_s^* = h_{p^s,1}\in X(\Phi)(L_s^*).
$$
and define
$$
\tilde{h}^*_s=(E_{p^s,1}, \tilde{\mathfrak{n}}_{p^s,1}, \pi_{p^s,1}) 
 \in X(\tilde{\Phi})(L_s^* ),
$$
where
$$
\tilde{\mathfrak{n}}_{p^s,1}=E_{p^s,1}[\mathfrak{DN}\cap \co_{p^{s+1}}]
\hspace{1cm}
\tilde{\Phi}=\Gamma_0(DN)\cap \Gamma_1(p),
$$  
so that  under the forgetful degeneracy map $X(\tilde{\Phi})\map{}X(\Phi)$ we have $\tilde{h}^*_s\mapsto h^*_s$.

Fix a finite set $\mathcal{S}$ of degree two primes of $K$ and assume that for every $v\in \mathcal{S}$, $\mathrm{Norm}(v)\equiv 1\pmod{Np}$.  Note that no $v\in \mathcal{S}$ divides $D$, and all $v\in \mathcal{S}$ are split completely in $L^*_\infty$. For each $v\in \mathcal{S}$ fix a place $\overline{v}$ of $\overline{\Q}$ above $v$,  let $\F_v$ be the residue field of $v$, and let  $X^\mathrm{ss}_v(\tilde{\Phi})$  denote the set of supersingular points on the reduction of $X(\tilde{\Phi})$ at $v$.  Our  assumption that $\mathrm{Norm}(v)\equiv 1\pmod{p}$  implies that all supersingular points have residue field $\F_v$ (if $\ell$ is the prime below $v$ then the square of the absolute Frobenius acts as $\langle \ell\rangle=\langle \pm 1\rangle$ on the supersingular locus, as in the proof of Proposition \ref{congruence}). The set $X^\mathrm{ss}_v(\Phi)$ is defined similarly.  Given $v\in\mathcal{S}$ and any $\mathcal{G}^*$-conjugate  $h$ of $\tilde{h}^*_s$,  we may define the \emph{$v$-reduction} of $h$
\begin{equation}\label{wee reduction}
\mathrm{red}_{v}(h)\in X^\mathrm{ss}_v(\tilde{\Phi}),
\end{equation}
to be the reduction at $\overline{v}$ of $h$.  Define
$$
\mathrm{Red}_v(h)\in X_v^\mathrm{ss}(\tilde{\Phi})^\mathcal{R}
$$
to be the $|\mathcal{R}|$-tuple with $\mathrm{red}_v(h^\sigma)$ in the $\sigma$ component for each $\sigma\in \mathcal{R}$.

\begin{Thm}[Cornut-Vatsal]\label{cv}
Let $G\subset \mathcal{G}^*$ be a compact open subgroup and let  $G\tilde{h}^*_s$ be the $G$-orbit of $\tilde{h}_s^*$. For all $s\gg 0$  the function
$$
G\tilde{h}_s^* \map{\prod_{v\in\mathcal{S}}\mathrm{Red}_v} 
 \prod_{v\in\mathcal{S}}X^\mathrm{ss}_v(\tilde{\Phi})^\mathcal{R}
$$
is surjective.
\end{Thm}

\begin{proof}
This is a special case of \cite[Theorem 3.5]{C-V}.  The set  $\mathcal{R}$ satisfies the hypotheses of that theorem by the discussion of \cite[\S 3.3.2]{C-V}.
\end{proof}

For each positive divisor $d$ of $D$  there is a degeneracy map 
$$
\lambda_d: X(\tilde{\Phi})\map{}X(\Phi)
$$
defined on moduli by $(E,C,P)\mapsto (E/C', C'', f(P))$ where $C$ is a cyclic order $DN$ subgroup of $E$,  $P$ is a point of exact order $p$ of $E$, $C'\subset C$ is the unique order $d$  subgroup of $C$, $f:E\map{}E/C'$ is the quotient map, and $C''\subset f(C)$ is the unique order $N$ subgroup.

\begin{Lem}\label{galois degen}
Let $d$ be a positive divisor of $D$ and let  $\sigma\in\mathcal{R}_\ram$ be the Frobenius of the unique ideal of $\co$ of norm $d$.  Then $\lambda_d(\tilde{h}^*_s)= (h^*_s)^{ \sigma}.$
\end{Lem}

\begin{proof}
Let  $\mathfrak{D}_{p^{s+1}}=\mathfrak{D}\cap\co_{p^{s+1}}$ and $\mathfrak{N}_{p^{s+1}}=\mathfrak{N}\cap\co_{p^{s+1}}$, and let $x\in\widehat{K}^\times$ be a finite idele which is a uniformizer at every prime divisor of $d$ and has trivial component at all other primes.  In particular $\sigma$ is the Artin symbol of $x$. The effect of the degeneracy map $\lambda_d$ on $\tilde{h}^*_s$ is
\begin{eqnarray*}
\tilde{h}^*_s& \iso &(\C/\co_{p^{s+1}}, \ (\mathfrak{D}_{p^{s+1}} \mathfrak{N}_{p^{s+1}} )^{-1}/\co_{p^{s+1}}, \  p^s\varpi ) \\
& \mapsto  &   (\C/x^{-1}\co_{p^{s+1}}, \  (x  \mathfrak{N}_{p^{s+1}} )^{-1}/x^{-1}\co_{p^{s+1}}, \  p^s\varpi ).
\end{eqnarray*}
On the other hand, the main theorem of complex multiplication  provides an isomorphism of elliptic curves with $\Gamma_0(N)\cap\Gamma_1(p)$ level structure
\begin{eqnarray*}
(h^*_s)^\sigma  &\iso &  (\C/\co_{p^{s+1}}, \  \mathfrak{N}_{p^{s+1}}^{-1}/\co_{p^{s+1}}, \  p^s\varpi )^\sigma  \\
&\iso&  (\C/x^{-1}\co_{p^{s+1}}, \  (x  \mathfrak{N}_{p^{s+1}})^{-1}/x^{-1}\co_{p^{s+1}}, \  p^sx^{-1} \varpi ).
\end{eqnarray*}
As $x$ has trivial component at $p$, $p^s\varpi$ and $p^sx^{-1}\varpi$ determine the same element of 
$$
\widehat{K}/x^{-1}\widehat{\co}_{p^{s+1}} = K/x^{-1}\co_{p^{s+1}}\iso E^\sigma_{p^s,1}(\C)_\tors,
$$
and so $\lambda_d(\tilde{h}_s^*)=(h^*_s)^\sigma$.
\end{proof}

\begin{Prop}[Ribet]\label{P:ribet}
Fix a $v\in\mathcal{S}$ and let $M_v(\Phi)$ and $M_v(\tilde{\Phi})$  denote the $\Z$-modules of degree zero divisors on  $X^\mathrm{ss}_v(\Phi)$ and $X^\mathrm{ss}_v(\tilde{\Phi})$, respectively. The product of   degeneracy maps
$$
\prod_{d\mid D}\lambda_{d,*}:  M_v(\tilde{\Phi})  \map{}  \prod_{d\mid D}M_v(\Phi)
$$ 
is surjective.
\end{Prop}

\begin{proof}
By induction on the number of prime divisors of $D$ it suffices to  prove that the product of degeneracy maps
\begin{equation}\label{ribet degen display} 
M_v(\Gamma_0(ND'q)\cap\Gamma_1(p))  \map{}  M_v(\Gamma_0(ND')\cap\Gamma_1(p))
\times M_v(\Gamma_0(ND')\cap\Gamma_1(p))
\end{equation}
is surjective for any divisor $D' q \mid D$ with $q$ prime  (using the obvious generalization of the notation $M_v(\Phi)$). If one replaces $\Gamma_1(p)$ by $\Gamma_0(p)$ then this is precisely \cite[Theorem 3.15]{ribet}.  We give the full proof in the general case.

Let $A$ be a supersingular elliptic curve over $\F_v$ endowed with a $\Gamma_0(ND')$-level structure, $C$, and a $\Gamma_1(p)$-level structure,  $P\in A[p]$.  The first  claim is that for any $t\in(\Z/p\Z)^\times$ there is an endomorphism of $A$ which has degree an odd power of $q$,  restricts to an isomorphism of $C$, and takes $P$ to $tP$. Indeed let 
$$
S=\{ f\in\End_{\F_v}(A) \mid f(C)\subset C, f(P)\in\Z\cdot P \}.
$$
Then $S$ is a level $ND'p$ Eichler order in a rational quaternion algebra $B$ of discriminant $\ell=\mathrm{char}(\F_v)$.  Set $\widehat{S}=S\otimes_\Z \widehat{\Z}$ and define $\widehat{B}$ similarly.  The action of $S\otimes\Z_p$ on the subgroup generated by $P$ determines a surjective character 
$$
\psi: \widehat{S}^\times\map{}(S\otimes\Z_p)^\times  \map{}(\Z/p\Z)^\times
$$
whose kernel is a compact open subgroup  $U\subset \widehat{B}^\times$.   Pick 
$\tau \in \widehat{S}^\times $  such that $\psi(\tau)=t$ and such that  $\tau$ has  trivial components away from $p$.   Using strong approximation as in \cite[Proposition 3.1]{Jordan-Livne} we may write $\tau=\beta u b$ in which $\beta\in B^\times$, $u\in U$ has trivial component at $q$, and $b\in\widehat{B}^\times$ has trivial components away from $q$ and reduced norm satisfying $\ord_q(\mathrm{N}( b))=-1$.  These conditions imply that $\beta$ lies in the  $\Z[1/q]$-order $S[1/q]$, has reduced norm $q$, and satisfies $\psi(\beta)=\psi(\tau)=t$.  Multiplying $\beta$ by a sufficiently large power  of $q$ gives an element  $f=q^e\beta \in S$ having reduced norm ($=$ degree) an odd power of $q$, and we are free to assume that $q^e\equiv 1\pmod{p}$.  But then
$$
f(P)= \psi( q^e \beta) \cdot P = \psi(\beta)\cdot P=\psi(\tau) \cdot P=  t\cdot P.
$$  
As the degree of $f$ is prime to $ND'$, $f$ restricts to an isomorphism of $C$ and so satisfies the desired properties.

Now suppose we are given two supersingular elliptic curves $A$ and $A'$ over $\F_v$, each endowed with a $\Gamma_0(ND')\cap\Gamma_1(p)$-level structure.  According to \cite[Lemma 3.1.7]{ribet} there is an isogeny $A\map{}A'$ having degree an odd power of $q$ and preserving the $\Gamma_0(ND'p)$-level structures.  Pre-composing this isogeny with an endomorphism of $A$ as in the preceding paragraph we obtain an isogeny $A\map{}A'$ of degree an \emph{even} power of $q$ which takes the $\Gamma_0(ND')\cap\Gamma_1(p)$-level structure on $A$ isomorphically to that on $A'$.  We may factor this isogeny as
$$
A_0\map{\pi_0} A_1 \map{\pi_1}\cdots \map{} A_{2i-1}\map{\pi_{2i-1}}A_{2i}
$$
in which $A=A_0$, $A'=A_{2i}$, and each $\pi_j$ has degree $q$.  We define  $\Gamma_0(ND')\cap\Gamma_1(p)$-level structures on each $A_j$ in such a way that the $\pi_j$'s preserve the level.  Each $\pi_j$ and its dual $\pi_j^\vee$ then define $\Gamma_0(q)$-level structures on $A_j$ and $A_{j+1}$, respectively.   Abusing notation, we write $\pi_j$ to indicate $A_j$ with its $\Gamma_0(ND')\cap \Gamma_1(p)$-level structure, together with its $\Gamma_0(q)$-level structure determined by $\pi_j$.  Similarly we write $\pi_j^\vee$ for $A_{j+1}$ with its $\Gamma_0(ND')\cap \Gamma_1(p)$-level structure together with its $\Gamma_0(q)$-level structure determined by $\pi_j^\vee$.  Then 
$$
\pi_0-\pi_1^\vee+\pi_2-\pi_3^\vee+\cdots+\pi_{2i-2}-\pi^\vee_{2i-1}
$$
is an element of $M_v(\Gamma_0(ND'q)\cap\Gamma_1(p))$, and a simple calculation shows that the image of this element under (\ref{ribet degen display}) is the pair $(A-A',0)$, where we view $A$ and $A'$ (with their added level structures) as divisors on $X(\Gamma_0(ND')\cap\Gamma_1(p))$.  Replacing $\pi_j$ by its dual everywhere gives a similar element of $M_v(\Gamma_0(ND' q)\cap\Gamma_1(p))$ whose image under  (\ref{ribet degen display}) is $(0,A-A')$.  As elements of this form generate the right hand side of (\ref{ribet degen display}), we are done.
\end{proof}

\begin{Thm}[Ihara]\label{T:ihara}
Let $J_1$ be the Jacobian of $X(\Phi)=X_1$ and write $\underline{J}_1$ for the N\'eron model of $J_1$ over $\Spec(\Z)$.  For every $v\in \mathcal{S}$ the cokernel of the natural map  $M_v(\Phi)\map{}\underline{J}_1(\F_v)$ has order dividing $\phi(N)$.
\end{Thm}

\begin{proof}
This is similar to \cite[Proposition 3.6]{prasad} and  to \cite[\S 4.2]{howard-cycles}, and we give a sketch of the proof. In all that follows, all geometric objects are defined over $\F_v$ unless otherwise indicated. The cokernel of the map in question is isomorphic to the Galois group of the maximal unramified abelian extension $C/X(\Phi)$ in which all supersingular points split completely. Ihara \cite{ihara} defines a curve $X_\mathrm{Ih}(Np)$ whose base change to $\overline{\F}_v$ is isomorphic to a connected component of the modular curve $X(Np)_{/\overline{\F}_v}$ classifying elliptic curves with full  $\Gamma(Np)$-level structure.  Owing to our assumption that  $\mathrm{N}(v)\equiv 1\pmod{Np}$, the  modular curve $X(Np)$ has all of its geometric components already defined over $\F_v$, and Ihara's curve is isomorphic (over $\F_v$) to every component. Fixing one such isomorphism, Ihara's curve admits a  degeneracy map to $X(\Phi)$.  By the main result of \cite{ihara} all supersingular points on $X_\mathrm{Ih}(Np)$ are defined over $\F_v$ and there are no  unramified extensions of $X_\mathrm{Ih}(Np)$ in which  all supersingular points split completely.   It follows that $C$ is a subextension of $X_\mathrm{Ih}(Np)$, and being abelian $C$ is then a subextension of $X_1(Np)$.  The theorem follows.
\end{proof}

Let  $\mathrm{Div}^0(\mathcal{G}^* h^*_s)$ denote the group of degree zero divisors on $X(\Phi)_{/L_s^*}$  which are supported on the  $\mathcal{G}^*$ orbit of $h_s^*$. For each $v\in\mathcal{S}$ define 
$$
\mathrm{Red}_v:J_1(L_s^*)\map{} \underline{J}_1(\F_v)^{\mathcal{R}_\tors}
$$
to be the map taking $P\in J_1(L_s^*)$ to the $|\mathcal{R}_\tors|$-tuple having the reduction at $\overline{v}$ of $P^\sigma$ in the $\sigma\in \mathcal{R}_\tors$ component.

\begin{Lem}\label{some surjectivity}
For $s\gg0$ the composition 
\begin{equation}\label{rnv I}
\mathrm{Div}^0(\mathcal{G}^*h^*_s)  \otimes\co_F\map{}  J_1(L^*_s )\otimes\co_F  \map{ \oplus_{v\in\mathcal{S}} \mathrm{Red}_v}  \oplus_{v \in \mathcal{S} }  \underline{J}_1(\F_v)^{\mathcal{R}_\tors} \otimes\co_F
\end{equation}
is surjective.
\end{Lem}

\begin{proof}
Let $\mathrm{Div}^0(\mathcal{G}^*\tilde{h}^*_s)$ be the group of degree zero divisors of $X(\tilde{\Phi} )_{/ L_s^* }$  supported on the $\mathcal{G}^*$ orbit of $\tilde{h}^*_s$. Identify $\mathcal{R}_\tors=\mathcal{R}\times\mathcal{R}_\ram$ and identify $\mathcal{R}_\ram$ with the set of positive divisors of $D$ by taking  $d\mid D$ to the Frobenius of the unique $\co$-ideal of norm $d$.
 Consider the  diagram
$$
\xymatrix{
{\mathrm{Div}^0(\mathcal{G}^* \tilde{h}^*_s)}  \ar[r] \ar[d]^{\lambda_{1,*}}&  
{\mathrm{Div}^0(\mathcal{G}^* \tilde{h}^*_s)}^{\mathcal{R}}  \ar[r] \ar[d]^{\prod_{d\mid D}\lambda_{d,*}}
&
{\bigoplus_{v\in \mathcal{S}}    M_v(\tilde{\Phi} )^{\mathcal{R}} }  \ar[d]^{\prod_{d\mid D}\lambda_{d,*}} \\
{\mathrm{Div}^0(\mathcal{G}^* h_s^*  )}  \ar[r] &
{\mathrm{Div}^0(\mathcal{G}^* h^*_s)}^{\mathcal{R}_\tors}  \ar[r]
&
{\bigoplus_{v\in \mathcal{S}} M_v(\Phi)^{\mathcal{R}_\tors}  } 
}
$$
in which the upper left horizontal arrow takes $h$ to the $|\mathcal{R}|$-tuple $(h^\sigma)_{\sigma}$ and similarly for the bottom left horizontal arrow. The upper right  horizontal arrow is the map on divisors induced by (\ref{wee reduction}) and the lower right arrow is defined in a similar way.  The composition of the upper row is then the map on divisors induced by the function of Theorem \ref{cv}, and so is surjective for $s\gg 0$ by that theorem.  The commutativity of the left hand square follows from Lemma \ref{galois degen}, while the commutativity of the right hand square is clear.  The rightmost vertical arrow is surjective by Proposition \ref{P:ribet}. Hence the bottom horizontal composition is surjective for $s\gg 0$, and the surjectivity of (\ref{rnv I}) then follows from Theorem \ref{T:ihara} and our hypothesis that $p$ does not divide $\phi(N)$.
\end{proof}

\begin{proof}[Proof of Theorem \ref{RNV}]
By  \cite[Lemma 8.1]{hida}  $\Ta_p^\ord(J_1)$ is the module of $\Gamma$ coinvariants of $\mathbf{Ta}^\ord$. This allows us to identify 
\begin{equation}\label{residual galois iso}
\Ta_p^\ord(J_1)/\mathfrak{m} \Ta_p^\ord(J_1) \iso\FT/\mathfrak{m}\FT.
\end{equation}
Define $\mathcal{H}_s$ to be the image of the composition
$$
\mathrm{Div}^0(\mathcal{G}^* h_s^*)\otimes\co_F  \map{}  J_1(L_s^* ) \otimes\co_F 
\map{}  H^1(L_s^* , \Ta^\ord_p(J_1))
\map{}  H^1(L_s^* , \FT/\mathfrak{m}\FT)
$$
and define
$$
e_\chi':H^1(L_s^* , \Ta^\ord_p(J_1))\map{}H^1(L_s^* , \Ta^\ord_p(J_1))
$$
by
$$
e_\chi'(h)=\sum_{\sigma\in\mathcal{R}_\tors}(\Theta^{-1}\chi)(\sigma)\cdot h^\sigma.
$$
We claim that the dimension of  $e_\chi'\mathcal{H}_{s}$ as an $R/\mathfrak{m}$-vector space goes to $\infty$ as $s\to\infty$. We may compute the image of $e_\chi'\mathcal{H}_s$ under the sum of localization maps
$$
H^1( L_s^* , \FT/\mathfrak{m}\FT)  \map{}\bigoplus_{v\in\mathcal{S}} H^1( L_{s,v}^*, \FT/\mathfrak{m}\FT)
$$
as the image of the composition  of (\ref{rnv I}) and
\begin{eqnarray}\lefteqn{ \label{VNV composition}
\bigoplus_{v\in\mathcal{S}}  \underline{J}_1(\F_v)^{\mathcal{R}_\tors}\otimes\co_F \map{} \bigoplus_{v\in\mathcal{S}} H^1(\F_v, \Ta_p^\ord(J_1))^{\mathcal{R}_\tors} } \hspace{1cm} \\
& & \map{} \bigoplus_{v\in\mathcal{S}}  H^1(\F_v, \FT/\mathfrak{m}\FT)^{\mathcal{R}_\tors} 
\map{e'_\chi}  \bigoplus_{v\in\mathcal{S}} H^1(\F_v, \FT/\mathfrak{m}\FT), \nonumber
\end{eqnarray}
using the isomorphism
$$
 H^1(\F_v, \FT/\mathfrak{m}\FT)\iso
  H^1_\mathrm{unr}(L^*_{s,v}, \FT/\mathfrak{m}\FT).
$$
In the above composition  the arrow labeled $e_\chi'$ preserves the $v$ component for each $v\in\mathcal{S}$ and  takes the $| \mathcal{R}_\tors|$-tuple $(x_\sigma)_{\sigma\in\mathcal{R}_\tors}$ to 
$$
\sum_{\sigma\in\mathcal{R}_\tors} (\Theta^{-1}\chi)(\sigma)\cdot x_\sigma \in H^1(\F_v, \FT/\mathfrak{m}\FT).
$$
 Each arrow in the composition (\ref{VNV composition}) is surjective; this follows from the fact that the absolute Galois group of $\F_v$ has cohomological dimension $1$, and, for the first arrow,  Lang's  theorem on the vanishing of $H^1(\F_v,\underline{J}_1(\overline{\F}_v))$ as in \cite[Proposition I.3.8]{milne:duality}.
 Combining this with Lemma \ref{some surjectivity}, for $s\gg 0$ the dimension of  $e'_\chi\mathcal{H}_{s}$ is at least that of  $\bigoplus_{v\in\mathcal{S}} H^1(\F_v, \FT/\mathfrak{m}\FT)$. If we choose $\mathcal{S}$ in such a way that the Frobenius of the prime of $\Q$ below each $v\in\mathcal{S}$ acts as (a conjugate of) complex conjugation on the extension of $K(\mu_{Np})$  cut out by $\FT/\mathfrak{m}\FT$, then  
 $$
 \dim_{R/\mathfrak{m}}H^1(\F_v, \FT/\mathfrak{m}\FT)=2.
 $$ 
 Thus the dimension of $e_\chi'\mathcal{H}_{s}$ over $R/\mathfrak{m}$ is at least $2|\mathcal{S}|$ for $s\gg 0$.  As we may take $\mathcal{S}$ as large as we like, $\dim(e_\chi'\mathcal{H}_{s})\to\infty$.

Recall $L_s^*=H_{p^{s+1}}(\mu_p)=L_{p^s,1}$ and $h_s^*=h_{p^s,1}$.  The character $$\Theta:G_\Q\map{}(R/\mathfrak{m})^\times$$ is trivial over $\Q(\mu_p)$, and twisting the 
isomorphism (\ref{residual galois iso}) by $\Theta^{-1}$ gives an isomorphism
$$
H^1(L^*_{s-1} ,\FT^\dagger/\mathfrak{m}\FT^\dagger)\iso
H^1( L^*_{s-1}, \Ta_p^\ord(J_1)/\mathfrak{m}\Ta_p^\ord(J_1) )\otimes\zeta_1
$$
taking the image of $\mathfrak{X}_{p^s}$ to the image of $U_p^{-1}\mathfrak{X}_{p^s,1}$ (by Definition \ref{Def:big}).  Tracing through the constructions of \S \ref{ss:construction},  the image of $h_s^*$ under
$$
X(\Phi)(L_s^*) \map{}\mathrm{Div}(X(\Phi)_{/L_s^*} )\otimes\co_F \map{e^\ord} J_1(L_s^*)^\ord 
 \map{} H^1(L_s^*,\FT/\mathfrak{m}\FT)
 $$
 agrees with the Kummer image of $y_{p^s,1} \in J_1(L_s^*)^\ord $,  which is taken to $\mathfrak{X}_{p^s,1}$ under 
\begin{equation}\label{final step}
H^1(L_s^*,\FT/\mathfrak{m}\FT)  \map{\mathrm{cor}}  H^1(L_{s-1}^*,\FT/\mathfrak{m}\FT)  \map{   h\mapsto h\otimes\zeta_1}  H^1(L_{s-1}^*,\FT/\mathfrak{m}\FT)\otimes\zeta_1.
\end{equation}
If follows that (\ref{final step}) takes $\mathcal{H}_s$ to the restriction of $\mathrm{Heeg}_s$ to $L_{s-1}^*$.  But
$$
\mathrm{N}(h_s^*)= \mathrm{N}(h_{p^s,1})=\mathrm{N}(\alpha_*(h_{p^{s-1},2}))=U_p\cdot h_{p^{s-1},1}=U_p\cdot h_{s-1}^*.
$$
by (\ref{Euler I equ 2}), where all norms are from $L_s^*$ to $L_{s-1}^*$, and so $\mathrm{cor}(\mathcal{H}_s)=\mathcal{H}_{s-1}$. We conclude
$$
\mathcal{H}_{s-1}\otimes\zeta_1=\mathrm{res}_{H_{p^s}(\mu_p)/H_{p^s}}(\mathrm{Heeg}_s)
$$
and in particular
\begin{eqnarray*}
(e_\chi'\mathcal{H}_{s-1})\otimes\zeta_1&=&
\left(\sum_{\sigma\in\mathcal{R}_\tors}\chi(\sigma)\sigma\right)\cdot ( \mathcal{H}_{s-1}\otimes\zeta_1 ) \\
&=&  \mathrm{res}_{H_{p^s}(\mu_p)/H_{p^s}}(e_\chi \mathrm{Heeg}_s).
\end{eqnarray*}
As the dimension of $e'_\chi\mathcal{H}_{s-1}$ is unbounded as $s\to\infty$, so is that of $e_\chi\mathrm{Heeg}_s$.
\end{proof}

%%%%%%%%%%%%%%%%%%%%%%%%%%%%%%%%%%%%%%%%%%%%%%%%%%%%%%%%%%%%%%%%%%

\subsection{A  two-variable main conjecture}
\label{SS:MC}

%%%%%%%%%%%%%%%%%%%%%%%%%%%%%%%%%%%%%%%%%%%%%%%%%%%%%%%%%%%%%%%%%%

Let $D_s/K$ be the subfield of $D_\infty$ of degree $p^s$ over $K$.   There is a nonnegative integer $\delta$ such that the fixed field $(H_{p^{s+1}})^{\mathcal{G}_\tors}=H_{p^{s+1}}\cap D_\infty$  is equal to $D_{p^{s+\delta}}$ for all $s\gg 0$. If $p$ does not divide the class number of $K$ then $\delta=0$. Define, for any $s\ge 0$ and using (\ref{big nek green}) and Proposition \ref{Prop:greenberg},
$$
\mathfrak{Z}_s\in \tilde{H}^1_f(D_s,\FT^\dagger)
$$
to be the image of $U_p^{-t}\cdot \mathfrak{X}_{p^{t+1}}$  under corestriction
$$
\tilde{H}^1_f(H_{p^{t+1}},\FT^\dagger)\map{}  \tilde{H}^1_f(D_s,\FT^\dagger)
$$
for any $t$ large enough that $D_s\subset H_{p^{t+1}}$. By Proposition \ref{Prop:ESrelations} the class $\mathfrak{Z}_s$ does not depend on the choice of $t$, and as $s$ varies these classes are  norm compatible. We may therefore define a module
$$
\tilde{H}^i_{f,\Iw}(D_\infty,\FT^\dagger)   =   \mil  \tilde{H}^i_f(D_s,\FT^\dagger)
$$
over the ring $R_\infty=R[[\Gal(D_\infty/K)]]$ and define $\mathfrak{Z}_\infty\in \tilde{H}^1_{f,\Iw}(D_\infty,\FT^\dagger)$ by  
$$
\mathfrak{Z}_\infty  = \mil\mathfrak{Z}_s.
$$
By Corollary \ref{VNV}, $\mathfrak{Z}_\infty$ is not $R_\infty$-torsion. Indeed, if $\mathfrak{p}$ is any arithmetic prime of $R$ and $\mathfrak{P}$ denotes the kernel of the map $R_\infty\map{} R_\infty\otimes_R F_\mathfrak{p},$ then $\mathfrak{P}$ is a height one prime of $R_\infty$ at which $\mathfrak{Z}_\infty$ is locally nontrivial.   As there are infinitely many such $\mathfrak{P}$, $\mathfrak{Z}_\infty$ is nontorsion (exactly as in the proof of Lemma \ref{ring lemma}).

The following conjecture is an extension of the Heegner point main conjecture for elliptic curves formulated by Perrin-Riou \cite{pr}. Partial results toward Perrin-Riou's conjecture have been obtained by Bertolini \cite{bert} and the author \cite{HowA,HowB}.  Assume that $R$ is regular, so that $R_{\infty,\mathfrak{P}}$ is a discrete valuation ring for every height one prime $\mathfrak{P}$ of $R_\infty$.  If $M$ is a finitely generated torsion $R_\infty$-module, define the characteristic ideal
$$
\mathrm{char}(M)=\prod_{\mathfrak{P}} \mathfrak{P}^{\mathrm{length}(M_\mathfrak{P})}
$$
where the product is over height one primes of $R_\infty$.  If $M$ is not torsion then set $\mathrm{char}(M)=0$.

\begin{Con}\label{mc}
Assuming that $R$ is regular
$$
\mathrm{char}\big(\tilde{H}^1_{f,\Iw}(D_\infty,\FT^\dagger) \big/ R_\infty \mathfrak{Z}_\infty\big)^2
= \mathrm{char}\big( \tilde{H}^2_{f,\Iw}(D_\infty,\FT^\dagger)_{\tors}  \big)
$$
in which the subscript $\tors$ indicates the $R_\infty$-torsion submodule.
\end{Con}

\begin{Rem}
As we assume that $N$ is prime to $\mathrm{disc}(K)$, the branch $R$ of the Hida family does not have complex multiplication by $K$ (that is, $R$ does not arise by the construction of 
\cite[\S 7]{hida_congruence} from a Hecke character of $K$).  In this situation Nekov\'{a}\v{r}, using results of Cornut-Vatsal and duality theorems for Selmer complexes, proves \cite[Theorem 12.9.11]{selmer_complexes} that for $i=1,2$
$$
\mathrm{rank}_{R_\infty}\big( \tilde{H}^i_{f,\Iw}(D_\infty,\FT^\dagger) \big)=1.
$$
\end{Rem}

\begin{Rem}
Iwasawa main conjectures are more often expressed in terms of  the Pontryagin dual of a Selmer group attached to the discrete Galois module $\mathbf{A}^\dagger=\Hom_{\Z_p}(\FT^\dagger,\mu_{p^\infty})$.  Conjecture \ref{mc} may be reformulated in this manner using the isomorphism of Poitou-Tate global duality \cite[\S 0.13]{selmer_complexes}
$$
 \tilde{H}^2_{f,\Iw}(D_\infty,\FT^\dagger) \iso \Hom_{\Z_p} \big(\tilde{H}^1_f(D_\infty,\mathbf{A}^\dagger),\Q_p/\Z_p\big)
$$
for an appropriate extended Selmer group  $\tilde{H}^1_f(D_\infty,\mathbf{A}^\dagger)$.
\end{Rem}

%%%%%%%%%%%%%%%%%%%%%%%%%%%%%%%%%%%%%%%%%%%%%%%%%%%%%%%%%%%%%%%%%%%%%%

\subsection{The horizontal nonvanishing conjecture}
\label{SS:HNV}

%%%%%%%%%%%%%%%%%%%%%%%%%%%%%%%%%%%%%%%%%%%%%%%%%%%%%%%%%%%%%%%%%%%%%

Let $\mathfrak{X}_1$ be the big Heegner point of conductor $c=1$ as in Definition \ref{Def:big}, so that
$$
\mathfrak{Z}_0= \mathrm{Cor}_{H_1/K}(\mathfrak{X}_1) \in\tilde{H}^1_f(K,\FT^\dagger).
$$

\begin{Con}\label{HNV}
The cohomology class $\mathfrak{Z}_0$ is not $R$-torsion.
\end{Con}

The theory of Euler systems allows one to use the cohomology classes $\mathfrak{X}_c$ to bound the Selmer group of $V_\mathfrak{p}^\dagger$.

\begin{Thm}[Nekov{\'a}{\v{r}}]\label{ES bound}
Let $\mathfrak{p}\subset R$ be an arithmetic prime with trivial wild character and weight $r\equiv k+j\pmod{p-1}$ with $r>2$. If $\mathfrak{Z}_0$ has nontrivial image in $\tilde{H}^1_f(K,V_\mathfrak{p}^\dagger)$ then $\dim_{F_\mathfrak{p}}\tilde{H}^1_f(K,V_\mathfrak{p}^\dagger)=1$.
\end{Thm}

\begin{proof}
Our hypotheses imply that the modular form $g_\mathfrak{p}$ attached to $\mathfrak{p}$ has even weight and trivial character.   In particular by  (\ref{nek-green compare}) and (\ref{bloch-kato compare}) the Greenberg, Nekov{\'a}{\v{r}}, and Bloch--Kato Selmer groups for $V_\mathfrak{p}^\dagger$ all agree. By Lemma \ref{oldform}  the Galois representation $V_\mathfrak{p}^\dagger$ is  a self-dual twist of the representation associated to a newform of level $N$.  Specializing the Euler system  of big Heegner points to an Euler system for $V_\mathfrak{p}^\dagger$, the stated theorem follows from the results of \cite{nek:euler}.  The Euler system used by Nekov{\'a}{\v{r}} is different from the one constructed here (or at least the construction is different), but the proofs in \cite[\S 6--13]{nek:euler} only require the existence of some family of cohomology classes satisfying the Euler system relations of \S\ref{ss:esr} and lying in the Bloch--Kato Selmer group. We note that the hypothesis $p\nmid(r-2)!$ of \cite{nek:euler} is used only in the construction of the Euler system classes and not in the bounding of the Selmer group.
\end{proof}

\begin{Cor}\label{generic rank}
Assume Conjecture \ref{HNV}.  Then $\tilde{H}^1_f(K,\FT^\dagger)$  is a rank one $R$-module and 
$$
\mathrm{rank}_R \ \tilde{H}^1_f(\Q,\FT^\dagger)= \left\{ 
\begin{array}{ll}
1 & \mathrm{if\ }w=1 \\
0 & \mathrm{if\ }w=-1
\end{array}\right.
$$
where $w=\pm 1$ is as in Proposition \ref{functional equation}.
\end{Cor}

\begin{proof}
By \cite[Proposition 4.2.3]{selmer_complexes} and \cite[Theorem 8.3.19]{neuk} the extended Selmer group $\tilde{H}^i_f(K,\FT^\dagger)$ is finitely generated over $R$ for all $i$.
Let $\mathfrak{p}\subset R$ be an arithmetic prime. Using Lemma \ref{DVR}, Nekov{\'a}{\v{r}}'s theory  \cite[Proposition 12.7.13.4(i)]{selmer_complexes} provides an exact sequence of extended Selmer groups
$$
0\map{}\tilde{H}^1_f(K,\FT^\dagger)_\mathfrak{p}  /\mathfrak{p}\tilde{H}^1_f(K,\FT^\dagger)_\mathfrak{p} \map{} \tilde{H}^1_f(K,V_\mathfrak{p}^\dagger)  \map{} \tilde{H}^2_f(K,\FT^\dagger)_\mathfrak{p}[\mathfrak{p}]\map{}0
$$
induced by the exact sequence 
$$
0\map{}\FT^\dagger_\mathfrak{p}\map{\pi}\FT^\dagger_\mathfrak{p} \map{}V_\mathfrak{p}^\dagger\map{}0
$$
for any generator  $\pi\in \mathfrak{p}R_\mathfrak{p}$. It follows from Lemma \ref{ring lemma} that $\mathfrak{Z}_0$ has nontrivial image in $\tilde{H}^1_f(K,V_\mathfrak{p}^\dagger)$ for all but finitely many $\mathfrak{p}$. Theorem \ref{ES bound} now shows that  $\tilde{H}^1_f(K,V_\mathfrak{p}^\dagger)$ is one dimensional  for infinitely many arithmetic primes $\mathfrak{p}$. The finite generation of  $\tilde{H}^2_f(K,\FT^\dagger)$ implies that  $\tilde{H}^2_f(K,\FT^\dagger)_\mathfrak{p}$  has no $R_\mathfrak{p}$ torsion for all but finitely many arithmetic primes $\mathfrak{p}$. From the above exact sequence of extended Selmer groups and Nakayama's lemma we deduce that the $R$-module  $\tilde{H}^1_f(K,\FT^\dagger)$ is locally  generated by a single element at infinitely many arithmetic primes $\mathfrak{p}$. The torsion submodule of $\tilde{H}^1_f(K,\FT^\dagger)$ has finite  support, and so $\tilde{H}^1_f(K,\FT^\dagger)$ is locally free of rank one at infinitely many arithmetic primes. The existence of any one such prime implies that  $\tilde{H}^1_f(K,\FT^\dagger)$ has  rank one. Proposition \ref{parity} shows that complex conjugation acts as $w$ on $\mathfrak{Z}_0$, and so the second claim follows from the first.
\end{proof}

\bibliographystyle{plain}

\end{document}